\documentclass[12pt]{article}
\usepackage{graphicx}
\usepackage{amsmath,amsthm,amssymb,amsfonts,enumerate}
\usepackage{euscript,mathrsfs}
\usepackage{xcolor}
\usepackage{empheq}
\usepackage[left=1.5cm,right=1.5cm,top=2cm,bottom=2cm]{geometry}
\usepackage{color}
\usepackage{enumitem}

\usepackage{stmaryrd}
\allowdisplaybreaks

\usepackage{epstopdf} 
\usepackage{xcolor} 
\usepackage{graphicx}
\usepackage [colorlinks=true, linkcolor=blue, citecolor=blue, urlcolor=blue]{hyperref}

\usepackage[labelformat=simple]{subcaption}

\newcommand{\nc}{\newcommand}

\usepackage{tikz}
\usetikzlibrary{patterns}

\def\softd{{\leavevmode\setbox1=\hbox{d}%
          \hbox to 1.05\wd1{d\kern-0.4ex{\char039}\hss}}}

\catcode`\@=11 \@addtoreset{equation}{section}

\catcode`\@=12

\newcommand{\range}[2]{\{#1, \ldots, #2 \}}

\newcommand\TS{\Delta t}

\newcommand\dt{\,\mathrm{dt}}
\newcommand\dx{\,\mathrm{d}x}
\newcommand\dS{\,\mathrm{dS}(x)}

\newcommand\di{\mathrm{div}}
\newcommand\dix{\mathrm{div}_x}
\newcommand\dih{\mathrm{div}_\mesh}
\newcommand{\laph}{\Delta_h }
\newcommand\Up{\mathrm{Up}}
\newcommand\Upi{\mathrm{Up}^{(i)}}

\newcommand{\bUp}{\textbf{Up}}

\newcommand\divup{\di_\Up }

\newcommand\ha{h^\alpha}
\nc{\pdedge}{ \eth _\E }
\nc{\pdedgei}{ \eth _\E^{(i)} }
\nc{\pdedgej}{ \eth _\E^{(j)} }

\nc{\pd}{ \partial }
\nc{\pdmesh}{ \pd _\mesh }
\nc{\pdmeshi}{ \pd _\mesh ^{(i)} }
\nc{\pdmeshj}{ \pd _\mesh ^{(j)} }

\nc{\shkl}{\sum_{\sigma =K|L\in \faces}}

\nc{\Hc}{ \mathcal{H} }

\newcommand{\Xspace}{X_\mesh}
\newcommand{\Yspace}{ {\bf Y_\edges}}
\newcommand{\Yspacei}{Y_{i,\edges}}

\newcommand{\PiE}{ \Pi _ \E}
\newcommand{\PiEi}{ \Pi _ \E^{(i)}}

\newcommand{\Gradedge}{ \nabla _\E}

\newcommand{\Gradedged}{ \nabla _\epsilon}

\newcommand{\Grad}{ \nabla _x}

\newcommand{\Lapm}{ \Delta_{\mesh}}
\newcommand{\Lape}{ \Delta_{\edges}}
\newcommand{\Lapmi}{ \Delta_{\mesh}^{(i)}}

\nc{\mainspace}{X(\mathcal{T})}
\nc{\dualspace}{Y(\mathcal{E})}
\nc{\norm}[1]{\lVert#1\rVert}

\nc{\co}[1]{\text{co}\{#1\} }

\nc{\up}{\rm up}
\nc{\nup}{n,{\rm up}}
\nc{\udn}{\vu_h^n \cdot \bfn_{\sigma,K}}
\nc{\Udn}{\vU_h^n \cdot \bfn_{\sigma,K}}
\nc{\Udnup}{ \avU_h^{\nup} \cdot \bfn_{\sigma,K}}

\newcommand{\avs}[1]{\left\{\!\!\left\{ #1\right\}\!\!\right\}}
\newcommand{\avsi}[1]{\left\{\!\!\left\{ #1\right\}\!\!\right\}^{(i)}}

\newcommand{\ve}{\mathbf{e}}
\newcommand{\vei}{\ve_i}

\newcommand\vr{\varrho}
\newcommand\vn{\mathbf{n}}
\newcommand\vu{\mathbf{u}}
\newcommand\uu{\vu}

\newcommand\vv{\mathbf{v}}

\newcommand\vx{\mathbf{x}}

\newcommand\vrh{\vr_h}
\newcommand\vuh{\vu_h}
\newcommand\vvh{\vv_h}

\newcommand\vUh{\vU_h}

\newcommand\sumi{\sum_{i=1}^d}
\newcommand\sumj{\sum_{j=1}^d}

\nc{\bS}{ \mathbb{S}}
\newcommand{\bI}{\mathbb{I}}

\newcommand\mesh{\mathcal{T}}
\newcommand\faces{\mathcal{E}}
\newcommand\edges{\faces}
\newcommand\edgesK{\faces(K)}
\newcommand\edgesi{\faces_i}
\newcommand\edgesiK{\faces_i(K)}

\newcommand\edge{\sigma}

\newcommand\facei{\mathcal{E}_i}

\newcommand\edgesint{\faces}

\nc{\E}{\mathcal{E}}
\nc{\Di}{\mathcal{D}_i}

\newcommand{\neighdual}{ \mathcal{N}^\star }

\newcommand\sumfaceK{\sum_{\sigma \in \edgesK}}
\newcommand\sumneigh{\sum_{L \in \mathcal{N}(K)}}
\newcommand\sumK{\sum_{K \in \mesh}}
\newcommand\sumfaceint{\sum_{\sigma \in \faces}}

\newcommand{\stiei}{\sum_{\sigma \in \edgesi }}

\nc{\intDsi}[1]{ \stiei \int_{D_\sigma}#1\dx}
\nc{\inttime}{ \int_0^T}

\newcommand\aleq{\lesssim}

\newcommand{\vih}{ v_{i,h} }
\newcommand{\uih}{ u_{i,h} }
\newcommand{\Uih}{ U_{i,h} }
\newcommand{\avih}{ \Ov{v}_{i,h} }
\newcommand{\auih}{ \Ov{u}_{i,h} }
\newcommand{\aujh}{ \Ov{u}_{j,h} }

\nc{\gradu}{\nabla \vu}
\newcommand{\myangle}[1]{\langle#1\rangle}
\newcommand{\DC}{C^\infty_c}

\nc{\bn}{\vn}
\nc{\bfr}{\mathbf{r}}
\nc{\bfn}{\vn}
\nc{\bfu}{\vu}
\nc{\bfv}{\vv}
\nc{\bfx}{\vx}
\nc{\avu}{\overline{\vu}}
\nc{\vU}{\mathbf{U}}

\nc{\avU}{\overline{\mathbf{U}}}

\newcommand{\Nu}{\mathcal V}
\newcommand{\sigmap}{\sigma _{K\text{,} i +}}
\newcommand{\sigmam}{\sigma _{K\text{,} i -}}

\nc{\abs}[1]{\left\lvert#1 \right\rvert}
\nc{\jump}[1]{\left\llbracket#1\right\rrbracket}

\nc{\pdt}{D_t}
\nc{\Pim}{\Pi_{\mesh}}
\nc{\Pie}{\Pi_{\edges}}

\newcommand{\Ov}[1]{\overline{#1}}
\newcommand{\bfphi}{\boldsymbol{\phi}}
\newcommand{\bfPhi}{\boldsymbol{\Phi}}
\nc{\intOh}{\int_{\Omega}}
\nc{\intO}[1]{\int_{\Omega} #1 \dx} 
\nc{\intS}[1]{\int_{\sigma} #1 \dS}
\nc{\intSh}[1]{\int_{\sigma} #1 \dS}
\nc{\intK}[1]{\int_{K} #1 \dx} 

\nc{\divx}{\div_x}

\nc{\ceil}[1]{\lceil#1\rceil}
\nc{\floor}[1]{\lfloor#1\rfloor}

\newtheorem{thm}{Theorem}[section]
\newtheorem{lemma}[thm]{Lemma}
\newtheorem{coro}[thm]{Corollary}
\newtheorem{defi}[thm]{Definition}
\newtheorem{prop}[thm]{Proposition}
\newtheorem{remark}{Remark}
\usepackage[normalem]{ulem}
\usepackage{cancel}
\nc{\mycolor}{\color{magenta}}
\nc{\rdele}{\color{brown}\sout}
\nc{\cblue}[1]{\textcolor{blue}{#1}}
\nc{\cred}{\color{red}}
\nc{\cmag}{\color{magenta}}
\nc{\cbrown}{\color{brown}}
\nc{\mydele}[1]{\textcolor{brown}{\sout {#1}}}

\title{Convergence and error estimates for a finite difference scheme for the multi-dimensional compressible Navier--Stokes system}

\author{Hana Mizerov\'a \thanks{Both authors thank Czech Sciences Foundation (GA\v CR), Grant Agreement 18--05974S for supporting the research. 
The Institute of Mathematics of the Czech Academy of Sciences is supported by RVO:67985840. } $^{,\dagger}$
\and Bangwei She $^{*, \ddagger}$
}
\date{}

\begin{document}

\maketitle

\centerline{$^*$ Institute of Mathematics of the Czech Academy of Sciences}
\centerline{\v Zitn\' a 25, CZ-115 67 Praha 1, Czech Republic}
\centerline{ \{mizerova, she\}@math.cas.cz}

\bigskip
\centerline{$^\dagger$ Department of Mathematical Analysis and Numerical Mathematics}
\centerline{Faculty of Mathematics, Physics and Informatics of the Comenius University}
\centerline{Mlynsk\' a dolina, 842 48 Bratislava, Slovakia} 

\bigskip
\centerline{$^\ddagger$ Faculty of Mathematics and  Physics of the Charles University}
\centerline{Sokolovsk\' a 83, 186 75 Praha 8, Czech Republic} 

\bigskip

\begin{abstract} 
We prove  convergence of a finite difference approximation of the compressible Navier--Stokes system  towards the strong solution in  $R^d,$ $d=2,3,$ for the adiabatic coefficient $\gamma>1$. 
Employing the relative energy functional, we find a convergence rate  which is \emph{uniform} in terms of the discretization parameters for $\gamma \geq d/2$.
All results are \emph{unconditional} in the sense that we have no assumptions on the regularity nor boundedness of the numerical solution. 
We also provide numerical experiments to validate the theoretical convergence rate. 
To the best of our knowledge this work contains the first unconditional result on the convergence of a finite difference scheme for the unsteady compressible Navier--Stokes system in multiple dimensions.  

\end{abstract}


\tableofcontents 

\section{Introduction}\label{sec:1}
 We study the viscous compressible fluid flow problem described by the Navier--Stokes system
\begin{subequations}\label{strong_ns}
\begin{align}
\partial_t \vr + \dix(\vr\vu) &= 0,\label{strong_CON} \\
\partial_t (\vr \vu) + \dix(\vr \vu \otimes \vu) + \Grad  p(\vr) &= \dix \mathbb{S} ,  \label{strong_MOM}
\end{align}
\end{subequations}
in the time--space cylinder $[0,T] \times \Omega $, $\Omega \subset R^d, d=2,3$, where $\vr$ is the density, $\vu$ is the velocity field, and
$\bS$ is the viscous stress tensor given by
\[\bS = \mu (\Grad \vu + \Grad^T \vu) + \lambda \dix \vu \bI,\; \mu >0, \; \mu+\lambda\geq 0. \]
The pressure is assumed to satisfy the {\em isentropic} law 
\begin{equation}\label{assumption_p}
p=a \vr^\gamma, \; a>0, \; \gamma>1.
\end{equation}
The system \eqref{strong_ns} is complemented with the space--periodic boundary conditions and  initial conditions
\begin{equation}\label{initial_assumption}
\vr(0,\vx) = \vr_0 > 0, \ \vr_0 \in L^\gamma(\Omega)\cup L^2(\Omega), \qquad \vu(0,\vx) = \vu_0, \ \vu_0 \in W^{1,2}(\Omega;R^d).
\end{equation}

The global existence of weak solutions to the Navier--Stokes system \eqref{strong_ns} was proven in \cite{FNP, Lions}. We would like to point out that the existence results require $\gamma > d/2$ which excludes the diatomic gas for $\gamma=1.4$. 
Concerning the existence of strong solutions, we refer to the result of Valli \cite{valli} for sufficiently smooth initial data, 
see also more recent results in \cite{Cho, Sun}. 

Despite various efficient numerical schemes in literature~\cite{Cockburn, DoFei, GGHL, HJL}, the numerical convergence analysis is open in general.  To our best knowledge, there are only two unconditional convergent schemes ready for use for the multi-dimensional Navier--Stokes system \eqref{strong_ns}.  
The pioneering work was done by Karper \cite{Karper}, where a combined finite element--finite volume method was shown to converge to a suitable weak solution under the assumption of $\gamma>3$. Later, this constraint has been relaxed to $\gamma \in (\frac65,2)$ by Feireisl and Luk\' a\v cov\' a-Medvid'ov\' a~\cite{FL}, who indeed showed the convergence of Karper's scheme to a strong solution via a powerful tool -- the weak--strong uniqueness principle in the class of the dissipative measure--valued (DMV) solutions.  
Very recently, we have extended this idea to the convergence proof for a finite volume method for $\gamma \in (1,2)$, see \cite{FLMS_FVNS}. 

We would also like to mention the error estimates results by Gallou\"et et al.~\cite{GallouetMAC}, Liu~\cite{Liu1, Liu2} and Jovanovi\'{c}~\cite{Jovanovic} with assumptions either on the boundedness of the numerical solution or higher regularity of the smooth solution, whose existence hasn't been proved. The current paper shares some similarities with the result of~\cite{GallouetMAC} in the sense that we both work with staggered grid, upwind flux and error estimates of the numerical solutions. 
Nevertheless, the differences are obvious. First our numerical scheme is different compared to the reference method~\cite{GallouetMAC}. We use different mass lumping and our upwind flux is easier to implement. 
Moreover, our numerical scheme includes an additional artificial diffusion term which helps us to prove the unconditional convergence of the numerical solution. 
Further, we do not need any assumptions on the asymptotic behaviour of the pressure while the referefence~\cite{GallouetMAC} does.

The first aim of this paper is to show the convergence of a finite difference approximation towards a strong solution of the system \eqref{strong_ns} for any $\gamma >1.$  Since there is no result on the existence of strong solutions to \eqref{strong_ns} on a polygonal domain, we decided to analyze here the space--periodic setting, i.e. the case of identifying the domain  with a flat torus, $\Omega = ([0,1]|_{0,1})^d,\ d=2,3.$
The main tool we employ is the DMV solution pioneered in \cite{FL}. Though it has been successfully applied to the convergence analysis of finite volume schemes for the compressible Euler and the Navier--Stokes(--Fourier) equations in our recent works \cite{FLM18,FLMS_FVNS,FLMS_FVNSF}, it is still a
non-trivial task
to apply it to the convergence analysis of other schemes and for a wider range of the adiabatic coefficient $\gamma$. The proof of convergence consists of two main steps:
\begin{itemize}[noitemsep,topsep=0pt,leftmargin=15pt]
\item deriving suitable stability estimates and consistency formulation to show that a sequence of numerical solutions generates a DMV solution of the limit system; 
\item employing the DMV weak--strong uniqueness principle to conclude the convergence of numerical solutions to  a strong solution of the limit system on the life span of the latter.
\end{itemize}
Let us emphasize that we do \emph{not} assume any boundedness nor additional regularity of approximate solutions other than those provided by the numerical scheme itself which makes our convergence result \emph{unconditional.} We also want to emphasize that the limit strong solution has been shown to exist for at least a short time interval, see~\cite{Cho}. 

The second aim is to investigate the error between the finite difference approximation of the compressible Navier--Stokes system \eqref{strong_ns} and the strong solution of the latter. Here we have to assume the same as \cite{GallouetMAC}, i.e., that there exists a strong regular solution in $C^2$ class.  We set a target of deriving  \emph{uniform} convergence rate in terms of the discretization parameters $\TS$ and $h$ for any $\gamma \geq d/2.$ 
The main tool we use to reach this goal is the discrete counterpart of the relative energy functional studied in \cite{FJN}, which reads 
\[
\mathfrak{E}(\vr, \vu| r, \vU) = \intO{\left( \frac12 \vr \abs{\vu- \vU}^2 + \mathbb{E}(\vr|r)  \right)},   \mbox{ with } \mathbb{E}(\vr|r)= \Hc(\vr) - \Hc'(r) (\vr -r ) -\Hc(r), 
\]
\[ \text{ and }\; 
\Hc(\vr) = \vr \int_1^{\vr} \frac{p(z)}{z^2} dz 
\;\text{ satisfying }\;  \vr \Hc'(\vr) -\Hc(\vr) =p(\vr), \; \Hc''(\vr) = \frac{p'(\vr)}{\vr}. 
\]
The relative energy functioal was designed for the analysis of distance between a suitable weak solution and the strong solution. Recently, this idea has also been used for the error analysis of numerical schemes, the distance between a numerical solution and a strong solution, see~\cite{GallouetMAC, Gallouet_mixed}. 
More precisely, we show the error estimates and the appropriate convergence rate following these four steps:
\begin{itemize}[noitemsep,topsep=0pt,leftmargin=15pt]
\item derive the discrete relative energy inequality which is inherent of the proposed numerical scheme;
\item approximate the discrete relative energy inequality with particularly chosen discrete test functions and suitably rewrite it into terms to be compared with the identity satisfied  by the strong solution;
\item show the identity (inequality) satisfied by the strong solution, the so-called consistency error;
\item apply Gronwall's inequality on the combination of the above two inequalities.
\end{itemize}

To the best of our knowledge this work contains  the first \emph{unconditional} convergence result for a finite difference approximation of the compressible Navier--Stokes system in multiple dimensions equipped with  \emph{uniform} convergence rate.  Despite the methodologies being already 
used in related works, the presented proofs remain highly non-trivial. Convergence analysis of the proposed numerical scheme requires elaborate treatment and technical estimates linked to the staggered grid and piecewise constant approximation of the discrete operators. 

The rest of the paper is organized as follows: in Section~\ref{sec:method} we introduce the numerical scheme,  necessary preliminaries, and the main results, i.e.  \emph{unconditional} convergence of the numerical solution and  \emph{uniform} convergence rate. In the next section we recall the energy stability and present the consequent uniform bounds, from which we prove the consistency formulation of the scheme.  Further, we prove the convergence of numerical solutions towards strong solution by employing the concept of DMV--strong uniqueness principle.
In Section 4,  we prove another main result on the error estimates. Section~5 is devoted to numerical experiments to support the theoretical results.
Concluding remarks come in the end.

\section{The numerical method and main results}\label{sec:method}

We start  by introducing the notations. 
We shall frequently use the notation $A \aleq B$ if $A \leq cB$, where $c >0$ is a constant that does \emph{not} depend on the discretization parameters $\TS$ and $h$. Moreover, $A 
\approx B$ means $A \aleq B$ and $B \aleq A$.
 We further write $c\in \co{a,b}$ if $\min(a,b)\leq c\leq\max(a,b)$.
In addition,  $\norm{\cdot}_{L^p}$ stands for $\norm{\cdot}_{L^p(\Omega)}$ and $\norm{\cdot}_{L^pL^q}$ stands for  $\norm{\cdot}_{L^p(0,T;L^q(\Omega))}$, since the time--space domain is fixed and this short notation shall bring no confusion. Finally, by $|\cdot|_{\rm max}$ we denote the maximum norm for continuous functions.

\subsection{Time discretization}\label{ss:time}
We divide the time interval $[0,T]$ into $N_t$ equidistant parts by a fixed time increment $\Delta t$ ($ =T/ N_t$). By $f_h^n$ we denote the value of a function $f_h$ at time  $t^n$  for $n\in \range{0}{N_t}$, where $h$ is the mesh parameter, see Section~\ref{sec_Space}. 
Then we use the backward Euler method to discretize the time derivative of a discrete function $f_h,$
\begin{equation*}
D_t f_h^n = \frac{ f_h^n - f_h^{n-1}}{\TS} \ \text{ for }
 n=1,2,\dots, N_t.
\end{equation*}
To prove the convergence of numerical solutions we will send the discrete parameters $h\approx\Delta t$ to zero and  investigate the weak limit of sequences of approximate functions in the $L^p$-setting ($p\geq 1$).  For this purpose we interpret quantities defined at the discrete points $t^n$ as piecewise constant functions with respect to the discretization of the time interval, 
\begin{equation}
f_h(t,\cdot) = f_h^0 \ \text{ for }\ t<\TS,\quad f_h(t)=f_h^n \ \text{ for }\ t\in [n\TS,(n+1)\TS), \; n\in\range{1}{N_t}.
\end{equation}
Consequently, we may write  
\begin{equation*}
D_t f_h = \frac{ f_h(t,\cdot) - f_h(t-\TS,\cdot)}{\TS}.
\end{equation*}
Note that we shall work with the approximations $\vrh,$ $\vuh$ and $p_h=p(\vrh)$ of the density, velocity and pressure, respectively.

\subsection{Space discretization}\label{sec_Space}

In order to introduce the finite difference MAC scheme we define the mesh and some  discrete operators.  

\subsubsection{Mesh}
\noindent{\bf \emph{Primary grid}}\quad The domain $\Omega$ is divided into 
compact  uniform quadrilaterals
\[\Omega =\bigcup_{K\in \mesh} K,\]
where the set of all elements, namely $\mesh$, forms the primary grid.  We denote by $\mathcal{E}(K)$ the set of all faces of an element $K,$ and by  $\edges$ the set of all faces of the primary grid $\mesh$.  
Further, we define 
\begin{align*}
\facei=\{\sigma \in \faces \; |  \; \sigma \text{ is orthogonal to } \vei  \}, \quad \edgesiK = \edgesK \cap \facei
\end{align*}
for any $i \in\range{1}{d}.$ Here $\vei$ stands for the unit basis vector of the canonical coordinate system.
We denote by $\bfx_K$ and $\bfx_\sigma$ the mass centers of an element $K \in \mesh$ and a face $\sigma \in \edges,$ respectively. 
By $h$ we denote the uniform size of the grid, meaning $|\bfx_K-\bfx_L|= h$ for any neighbouring elements $K$ and $L$. 
To distinguish the exact position of a face $\sigma \in \edgesK$  we may also use the notation $\sigma _{K\text{,} i \pm}$ if 
\begin{equation*}
\bfx_\sigma  =\bfx_ K \pm \frac{h}{2}\vei,\;  i \in \range{1}{d}.
\end{equation*} 
Further, let $\mathcal{N}(K)$ denote the set of all neighbouring elements of $K \in \mesh.$
For any $\sigma \in \faces$ adjacent to the element $K$ and its neighbour $L\in\mathcal{N}(K)$, we write $\sigma = K|L.$ 
We  denote $\sigma = \overrightarrow{K|L}$ if moreover $\bfx_L - \bfx_K = h\vei.$ Further $\vn_{\sigma, K}$ denotes the outer normal vector to a face $\sigma \in \edgesK.$ 


\medskip
\noindent {\bf \emph{Dual grid}}\quad Each face  $\sigma=K|L$ is associated to the  dual cell $D_\sigma =D_{\sigma,K} \cup D_{\sigma,L}$ which  is defined as the union of two half--cells, $D_{\sigma,K}$ and $D_{\sigma,L},$ adjacent to the face $\sigma=K|L,$ see Figure~\ref{fig-grid}. 
We set $\Di =\left\{D_\sigma  \; |  \; \sigma \in \edgesi\right\}$, $i \in \range{1}{d}.$ 
Note that for each fixed $i \in \range{1}{d}$ it holds 
\begin{equation*}\label{dualgrid}
\Omega=\bigcup_{\sigma\in {\E_i}} D_{\sigma}, \;\; {\rm int}(D_\sigma)\cap{\rm int}( D_{\sigma'})=\emptyset,\mbox{ for } \sigma,\sigma'\in {\E_i},\,\sigma\neq\sigma'.
\end{equation*}
Let  $\mathcal{N}^\star(\sigma)$ denote the set of all faces whose associated dual elements are the neighbours of $D_\sigma,$ i.e., 
\[ \neighdual(\sigma)=\{\sigma'\ |\  D_{\sigma'} \mbox{ is a neighbour of } D_\sigma\}.
\]

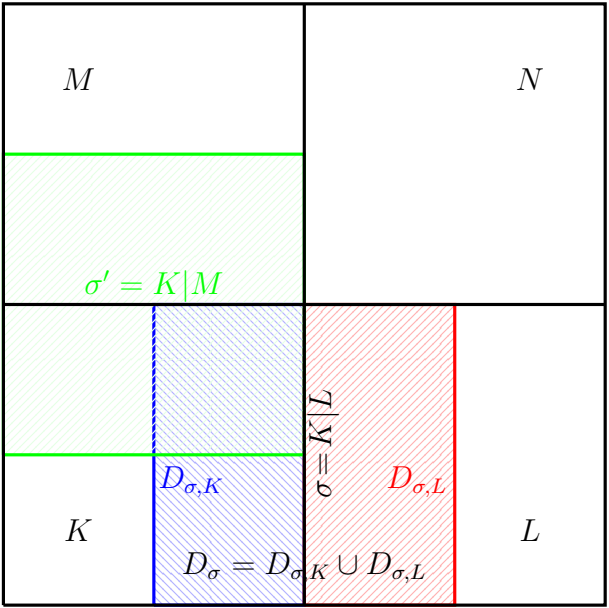
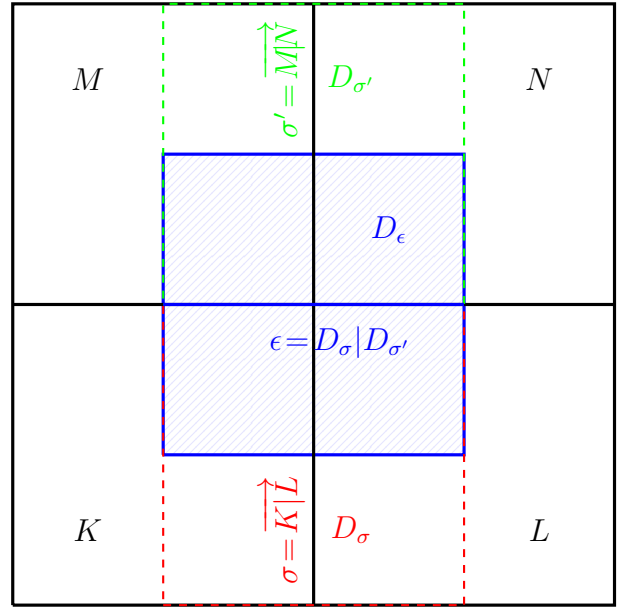
\begin{figure}[hbt]
\centering
\begin{subfigure}{0.45\textwidth}
\centering
\begin{tikzpicture}[scale=1]
\draw[-,very thick,blue=90!, pattern=north west lines, pattern color=blue!30] (2,0)--(4,0)--(4,4)--(2,4)--(2,0);
\draw[-,very thick,red=90!, pattern=north east lines, pattern color=red!30] (4,0)--(6,0)--(6,4)--(4,4)--(4,0);
\draw[-,very thick,green=10!, pattern=north east lines, pattern color=green!10] (0,2)--(4,2)--(4,6)--(0,6)--(0,2);
\path (1,1) node[] { $K$};
\path (7,1) node[] { $L$};
\path (4.25,2.15) node[rotate=90] { $\sigma\!=\!K|L$};
\path (2.5,1.64) node[] { \textcolor{blue}{$D_{\sigma,K}$}} ;
\path (5.5,1.64) node[] { \textcolor{red}{$D_{\sigma,L}$}};
\path (4,0.5) node[] { $D_\sigma=  D_{\sigma,K} \cup D_{\sigma,L}$};
\path (2,4.25) node[] { \textcolor{green}{$\sigma'=K|M$}};
\path (1.,7) node[] { $M$};
\path (7,7) node[] { $N$};%
\draw[-,very thick](0,0)--(8,0)--(8,8)--(0,8)--(0,0);
\draw[-,very thick](4,0)--(4,8);%
\draw[-,very thick](0,4)--(8,4);%
\end{tikzpicture}\caption{Dual grid in two dimensions}\label{fig-grid}
\end{subfigure}
\hfill
\begin{subfigure}{0.45\textwidth}
\centering
\begin{tikzpicture}[scale=1]
\draw[-,very thick,blue=10!, pattern=north east lines, pattern color=blue!10] (2,2)--(6,2)--(6,6)--(2,6)--(2,2);
\path (1,1) node[] { $K$};
\path (7,1) node[] { $L$};
\path (5,5) node[] { \textcolor{blue}{$D_\epsilon$}};
\path (4.5,7) node[] { \textcolor{green}{$D_{\sigma'}$}} ;
\path (4.5,1) node[] { \textcolor{red}{$D_\sigma$}};
\path (3.6,7) node[rotate=90]{ \textcolor{green}{$\sigma'\!=\!\overrightarrow{M\!|\!N}$}};
\path (3.6,1) node[rotate=90] { \textcolor{red}{$\sigma\!=\! \overrightarrow{K|L}$}};
\path (1.,7) node[] { $M$};
\path (7,7) node[] { $N$};%
\draw[-,very thick](0,0)--(8,0)--(8,8)--(0,8)--(0,0);
\draw[-,very thick](4,0)--(4,8);%
\draw[-,very thick](0,4)--(8,4);%
\draw[dashed, thick,green](2,4)--(6,4)--(6,8)--(2,8)--(2,4);%
\draw[dashed, thick,red](2,0)--(6,0)--(6,4)--(2,4)--(2,0);%
\draw[-,very thick,blue](2,4)--(6,4);%
\path (4.35,3.5) node[] { \textcolor{blue}{$\epsilon\!=\! D_\sigma | D_{\sigma'}$} };
\end{tikzpicture}\caption{Bidual grid in two dimensions}\label{bidual-grid}
\end{subfigure}
\caption{MAC grid in two dimensions}\label{MAC-grid}
\end{figure}

\paragraph{\em Bidual grid}
In order to perform a discrete analogue of integration by parts we need to define a suitable discrete gradient operator for the velocity. For this purpose, we introduce the dual face and the bidual grid as in \cite[Definition 2.1]{GallouetMAC}. 
Let $\widetilde{\E}(D_\sigma)$ denote the set of all faces of a dual cell $D_\sigma.$ A generic dual face and its mass center are denoted by $\epsilon \in \widetilde{\E}(D_\sigma)$ and  $\bfx_{\epsilon}$, respectively.  
For any $\epsilon$ which separates the dual cells $D_\sigma$ and $D_{\sigma'}$, we write 
$\epsilon = \overrightarrow{ D_\sigma | D_{\sigma'} }$ if $\bfx_{\sigma'} - \bfx_{\sigma} = h \vei $ 
for  $i \in \{ 1,\ldots, d\}$.  
Similarly  to the definition of the dual cell, a bidual cell $D_\epsilon$ associated to $\epsilon = D_\sigma | D_{\sigma'}$ is defined as the union of adjacent halves of 
$D_\sigma$ and $D_{\sigma'}$, see Figure~\ref{bidual-grid}. Finally, let $\widetilde{\E}$ be the set of all faces of the bidual grid, and 
$\widetilde{\E}_i$ be the set of all faces of the bidual grid that are orthogonal to $\vei$.

\subsubsection{Function spaces}\label{Fspace}
We work with the staggered grid. On one hand, we approximate the discrete density and pressure on the center of each element $K\in \mesh$ by $\vr_K$ and $p_K$, respectively.  On the other hand, we approximate the discrete velocity on the center of each edge $\sigma\in \edgesi$ by $v_{i,\sigma}$ for all  $i\in \range{1}{d}.$ 
For the purpose of analysis, it is more convenient to extend these quantities to functions defined in the domain $\Omega$. Thus, we define  
\begin{equation}\label{PCE}
\vrh(\vx) = \sum_{K \in \mesh} \vr_K 1_{K},\quad p_h(\vx) = \sum_{K \in \mesh} p_K 1_{K}, \quad \uih(\vx) = \sum_{\sigma \in \edgesi} u_{i,\sigma} 1_{D_\sigma},\quad \forall \; \vx\in\Omega, 
\end{equation} 
and introduce the following piecewise constant function spaces
\begin{equation*}
\begin{aligned}
\Xspace &=  \left\{  \phi \mid  \phi_h|_K = \text{ constant } \forall \; K\in \mesh \right\},\\
\Yspace & =\left(Y_{1,\edges},\ldots Y_{d,\edges} \right),  \quad \Yspacei =  \left\{  \phi \mid  \phi_h|_{D_\sigma} = \text{ constant } \forall \; \sigma \in \edgesi  \right\}, \; i\in \range{1}{d}.
\end{aligned}
\end{equation*}
where $1_K$ and $1_{D_\sigma}$  are the characteristic functions. 
Clearly, $\vrh, p_h \in \Xspace$ and $\vuh \in \Yspace$ and the following identities hold
\[ \int_\Omega r_h \dx = \sumK  |K| r_K, 
\quad \mbox{and} \quad  \int_\Omega \vuh \dx  = \sum_{i=1}^d \sum_{\sigma\in\edgesi} |D_\sigma|   u_{i,\sigma} \vei
\]
for $r_h$  being $\vrh$ or $p_h$.

\subsubsection*{Projection to the primary and dual grids}
We define the \emph{projection operators}  to the primary and the dual grid by
\begin{align*}
\Pim:& L^1(\Omega) \to \Xspace, &\  
\Pim \phi= \sum_{K\in \mesh} (\Pim \phi)_K  1_{K}, 
\qquad &\ (\Pim \phi)_K =  \frac{1}{|K|} \int_{K} \phi \dx, \\
\PiEi:& W^{1,1}(\Omega) \to \Yspacei, &\ 
\PiEi \phi = \sum_{\sigma \in \E}  (\PiEi \phi)_{\sigma} 1_{D_\sigma}, 
\qquad &\ (\PiEi \phi)_{\sigma} =   \frac{ 1}{|\sigma|} \intSh{ \phi }, 
\end{align*} 
respectively. Further, for any $\bfphi=(\phi_1, \ldots, \phi_d)$ we denote $\PiE \bfphi= \left(\Pie^{(1)}\phi_1, \ldots, \Pie^{(d)}\phi_d\right)$.

\subsubsection*{Interpolating discrete quantities between the grids}

For the proper implementation of the scheme we need to interpolate the functions defined on the primary grid to the dual grid and vice versa. To this end we define the \emph{average operator} for any scalar function $r_h \in \Xspace,$  
\[	\avs{r_h}_\sigma = \frac{r_K + r_L}{2},\ \sigma= K|L \in \edges.
\]
If in addition, $\sigma= K|L \in \edgesi$ for an $i\in\{1, \ldots, d\}$, we write
\begin{align*}
\avs{r_h}^{(i)}_\sigma = \frac{r_K + r_L}{2},  \mbox{ and } \avsi{r_h} = \sum_{\sigma \in \edgesi} 1_{D_\sigma}  \avsi{r_h}_\sigma .
\end{align*}
Further, for vector--valued functions $\bfr_h=(r_{1,h}, \ldots, r_{d,h}) \in \Xspace^d$ and $ \vvh=(v_{1,h},\ldots,v_{d,h}) \in \Yspace,$ we define 
\begin{align*}
\avs{\bfr_h} &= \left( \avs{r_{1,h}}^{(1)}, \ldots, \avs{r_{d,h}}^{(d)} \right),   \quad  
\\
\avih  = \sum_{K \in \mesh} 1_{K}  (\avih)_K,  \quad  (\avih)_K &= \frac{v_{i,{\sigmap}}+ v_{i,{\sigmam}}}{2}, \; \text{ and }
(\Ov{\vv}_h)_K  = \sumi (\avih)_K \vei.
\end{align*}

\subsubsection*{Difference operators}
As we are working with the staggered grid, we need to define the difference operators for all grids described above. 
Thus, for any $r_h \in \Xspace$ and $\vvh \in \Yspace$, we introduce the following \emph{discrete derivatives} for $ i\in \range{1}{d},$
\begin{equation}\label{der1}
\begin{split}
& \pdedgei r_h (\bfx) = \sum_{\sigma \in \edgesi}1_{D_\sigma} (\pdedgei r_h)_{\sigma} , \quad \mbox{where }  \ (\pdedgei r_h)_{\sigma}  = \frac{r_{L} - r_{K}}{h}, \ \sigma=\overrightarrow{ K|L}\in \edgesi, \; 
\\ &
 \pdmeshi \vih(\bfx) = \sumK 1_K (\pdmeshi \vih)_{K} , \quad \mbox{where } \ (\pdmeshi \vih)_{K} = \frac{v_i|_{\sigmap} - v_i|_{\sigmam}}{h}, \ K\in\mesh.
\end{split}
\end{equation}
Further, we extend the above notations to the definition of  the \emph{discrete gradient and divergence operators} 
\begin{equation}\label{div_grad}
\dih \vvh(\bfx)  = \sumi  \pdmeshi \vih (\bfx) \quad \mbox{and} \quad
  \Gradedge r(\bfx) = \left( \pdedge^{(1)}r, \ldots,  \pdedge^{(d)}r \right)(\bfx),
 \end{equation}
which are piecewise constant on the primary  and dual grid, respectively. 
It is easy to observe that 
\[(\dih \vvh)_K = \frac{1}{|K|}\sumfaceK \intSh{\vvh \cdot \bfn_{\sigma,K}}  .
\]
Next, we define the \emph{Laplace operators} for any $r_h\in\Xspace$ and  $\vvh\in\Yspace,$
\begin{equation*}
\begin{aligned}
& 
 \Lapm r_h(\bfx)  = \sum_{K\in\mesh}1_K (\Lapm r_h )_K,   \quad \mbox{where } \
(\Lapm r_h )_K =  \frac{1}{h^2} \sumneigh (r_L - r_K),  \\
& \Lape \vih(\bfx) =  \sum_{ \sigma \in \E}1_{D_\sigma} (\Lape \vih)_{\sigma}, \quad \mbox{where } \
  (\Lape \vih)_{\sigma} = \frac{1}{h^2} \sum_{\sigma' \in \neighdual(\sigma)} \big(v_{i,\sigma'} - v_{i,\sigma} \big). 
\end{aligned}
\end{equation*}
Denoting $\Lapmi r_h = \pdmeshi \left( \pdedgei r_h \right)$ we observe $\Lapm r_h =\sumi \Lapmi r_h $ for $r_h \in \Xspace.$

We specify the \emph{discrete velocity gradient} on the bidual grid which is different compared to the gradient operator defined in \eqref{div_grad}. Indeed, 
\[
\Gradedged \vv(\bfx) = \big( \nabla_\epsilon v_1(\bfx), \ldots, \nabla_\epsilon v_d(\bfx)\big) \quad \mbox{with }\
\Gradedged v_i(\bfx) = \big( \eth_1 v_i(\bfx), \ldots, \eth_d v_i(\bfx) \big),
\]
where
\[
\eth_j v_i(\bfx) = \sum_{\epsilon \in \widetilde{\E}_j } (\eth_j v_i)_{D_\epsilon} 1_{D_\epsilon} \ \mbox{ and }\; (\eth_j v_i)_{D_\epsilon} = \frac{v_{\sigma'} -v_{\sigma}}{h}, \mbox{ for } \epsilon = \overrightarrow{ D_\sigma | D_{\sigma'} } \in \widetilde{\E}_j \mbox{ and }
\sigma,  \sigma' \in \E_i.
\]

\subsubsection*{Upwind divergence}
We firstly remark the notation
\[
r^+ = \max\{0,r\}= \frac{1}{2} (r + |r|), \quad r^- = \min\{0,r\}=\frac{1}{2} (r -|r|), 
\]
and define an upwind function for $r_h \in \Xspace$ under a given velocity field $\vvh \in \Yspace$
\begin{align*}
r_h^{\up} = \sum_{ \sigma \in \E}r_\sigma^{\up} 1_{D_\sigma}, \quad 
r_\sigma^{\up}= \left\{ \begin{array}{ll}
r_K, & v_{i,\sigma} \geq 0 ,\\
r_L, & v_{i,\sigma} <0,  \\
\end{array}\right. \quad \sigma = \overrightarrow{ K|L}\in \facei, \; i \in \range{1}{d}.
\end{align*}
Then the \emph{upwind flux} is given by
\begin{equation*}
\Upi [r_h,\vvh](\bfx)= \stiei  \Upi[r_h,\vvh]_\sigma 1_{D_\sigma}, \quad
\Upi[r_h,\vvh]_\sigma =  r_\sigma^{\up} v_{i,\sigma} = r_K (v_{i,\sigma})^+ + r_L (v_{i,\sigma})^-.
\end{equation*}
Further, we introduce an \emph{upwind divergence} operator 
\begin{equation}\label{upwind_divergence}
\divup [r_h,\vvh](\bfx) = \dih \bUp[r_h,\vvh]
  \quad \mbox{with} \quad  \bUp[r_h,\vvh] =\left( \Up^{(1)}[r_h,\vvh], \ldots, \Up^{(d)}[r_h,\vvh] \right), 
\end{equation}
and easily  observe that 
\begin{align*}
\divup [r_h,\vvh]_K 
 =  \sumi \sum_{\sigma \in \edgesiK} \frac{|\sigma|}{|K|} \Upi [r_h,\vvh]_\sigma \vei \cdot \bfn_{\sigma,K}
=\frac{1}{|K|}\sum_{\sigma \in \edgesK} \intSh{ r_h^{\up}\vvh\cdot \bfn_{\sigma,K}} .
\end{align*}

\subsection{Preliminaries} 
In this section we introduce some preliminaries. 
First, we recall the inverse estimates from \cite[Lemma 2.3]{HS_MAC}. For $r_h \in \{\Xspace, \Yspace\}$ it holds 
\begin{equation} \label{inv_est}
\| r_h \|_{L^p} \lesssim h^{ d (\frac1p -\frac1q)  } \|r_h \|_{L^q} \ \mbox{ for any }\  1 \leq q \leq p \leq \infty.
\end{equation}
Next, by the scaling argument, we report the trace inequality, see~\cite[equation (2.26)]{FLNNS}
\begin{equation}\label{ineq_trace}
\norm{r_h}_{L^p(\pd K)} \leq h^{-1/p}\norm{r_h}_{L^p( K)} \mbox{ for any } p \in [1,\infty].
\end{equation}
Further, according to \cite[Lemma 2.5]{HS_MAC}   for  $r_h \in \Xspace, \; \vv_h \in \Yspace$ we have 
\begin{equation}\label{up_avg}
\Upi[r_h, \vv_h]_{\sigma} = \avsi{r_h} v_{\sigma} - \frac{h}{2} \abs{v_{\sigma}} (\pdedgei r_h)_{\sigma},  \ \sigma \in \edgesiK, \; K \in \mesh.
\end{equation}

Some useful estimates related to the projections onto the discrete function spaces are comprised in the following lemma.
\begin{lemma}\label{NC}
Let $\phi \in C^1$, $\Phi \in C^2$, then for $i\in\range{1}{d}$, we have 
\begin{subequations}
\begin{equation}\label{NC1}
\abs{ \pdedgei \Pim \phi }  \aleq \norm{\phi}_{C^1},  \quad 
 \abs{ \pdmeshi \PiEi \phi }  \aleq \norm{\phi}_{C^1},   \quad
 \pdedgej   \Ov{ \PiEi \phi}    \aleq \norm{\phi}_{C^1}, 
\end{equation}
\begin{equation}\label{NC2}
\norm{\Gradedge \Pim{\phi}  -  \Grad \phi}_{L^\infty} \aleq h \norm{\phi}_{C^2}, \quad
\norm{\Gradedged \Pie{\bfPhi}  -  \Grad \bfPhi}_{L^\infty} \aleq h \norm{\bfPhi}_{C^2},
\end{equation}
\begin{equation}\label{NC3}
\Lapm \Pim \Phi \aleq \norm{\Phi}_{C^2},  \quad \abs{ \pdmeshj  \pdedgej   \Ov{ \PiEi \Phi}    } \aleq  \norm{\Phi}_{C^2}, \quad \norm{D_t \avU_h}_{L^\infty L^\infty } \aleq \norm{\pdt \vUh}_{L^\infty W{1,\infty} }
\end{equation}
\end{subequations}
\end{lemma}
\begin{proof}
We will prove only the first inequalities of \eqref{NC1} and \eqref{NC2} as the rest can be done analogously. 
By the mean value theorem there exists $x^* \in \co{ x_L ,  x_K }$ such that for any $\sigma =\overrightarrow{K|L} \in \edgesi$ it holds
\[ (\pdedgei \Pim \phi )_\sigma  = \frac{ (\Pim \phi)_L - (\Pim \phi)_K  }{h} =\frac{\pd \phi}{ \pd {x_i}} (x^*).
\]
Therefore we get the first inequality of \eqref{NC1}, i.e., 
\[
\abs{ \pdedgei \Pim \phi }_\sigma \leq  \abs{ \frac{\pd \phi}{ \pd {x_i}} }_{\max}   \aleq \norm{\phi}_{C^1}.
\]
Similarly, we know that for any $x \in D_\sigma$ there exists $x^{**} \in \co {x, x^*}$, such that 
\[
\abs{ (\pdedgei \Pim \phi )_\sigma - \frac{\pd \phi}{\pd x_i}} = \abs{  \frac{\pd^2 \phi}{ \pd x_i^2}(x^{**}) } \leq  \abs{ \frac{\pd^2 \phi}{ \pd x_i^2 }}_{\max}   \aleq \norm{\phi}_{C^2},
\]
which indicates the first inequality of \eqref{NC2}. 
\end{proof}

It is easy to check the following integration by parts formulae, so we omit the proof here.
\begin{lemma}\cite[Lemma 2.1]{HS_MAC}\label{lem_ibp}
Let $r_h, \phi_h \in \Xspace,$ and $\vvh,$ $\bfPhi_h \in \Yspace.$ Then 
\begin{subequations}
\begin{equation}\label{pp}
-\intO{ \Lapm r_h \; \phi_h }
= \intO{ \Gradedge r_h  \cdot \Gradedge \phi_h} , \quad 
-\intO {\Lape \vvh \cdot \Phi_h }
= \intO {\Gradedged \vvh  : \Gradedged \Phi_h } ,
\end{equation} 
\begin{equation}\label{grad_div}
- \intO {\dih \vvh \; r_h  } 
= \intO{ \vvh \cdot \Gradedge r_h }, 
\quad  
- \intO {\divup[r_h, \vvh] \phi_h   }=  \sumi \intO{ \Upi[r_h, \vvh] \cdot \pdedgei \phi_h } . 
\end{equation}
\end{subequations}
\end{lemma}

The definition of the upwind divergence~\eqref{upwind_divergence} yields a simple corollary of Lemma~\ref{lem_ibp}.
 \begin{coro}\label{C:intdivup0}
Let $r_h \in \Xspace, \vvh  = [v_{1,h}, \cdots, v_{d,h}] \in \Yspace$. Then $\displaystyle \sumK \di_\Up [r_h,\vvh]_K = 0$.
\end{coro}

The next lemma provides identities necessary to derive the consistency formulation of the proposed scheme.
\begin{lemma}
It holds for $\bfPhi \in W^{1,1}(\Omega)$ that 
\begin{equation}\label{div_equiv}
\int_K  \dih \Pie \bfPhi  = \int_K  \dix \bfPhi, \quad  K \in \mesh. 
\end{equation}
\end{lemma}
\begin{proof}
From the definition of $\dih$, we know that 
\begin{align*}
& \intK{\dih \Pie \bfPhi } =  \sumi \sum_{\sigma \in \edgesiK }  |\sigma|  \PiEi \Phi_i \vei \cdot \bfn_{\sigma, K} 
=  \sumi \sum_{\sigma \in \edgesiK }  \intSh{  \Phi_i \vei \cdot \bfn_{\sigma, K} }
\\ &
=   \sum_{\sigma \in \edgesK }  \intSh{ \bfPhi   \cdot \bfn_{\sigma, K} }
=  \intK{  \dix  \bfPhi}.
\end{align*}
\end{proof}

\begin{lemma}
Let $r_h \in \Xspace$, $\vvh \in \Yspace$, $\phi\in C^2$. Then there hold
\begin{equation}\label{conv1}
\begin{aligned}
 \intO{\divup[r_h,\vvh] \Pim \phi } &=\intO{ r_h \vvh \cdot \Grad\phi }  
+\intO{ r_h \vvh \cdot \big(\Gradedge(\Pim \phi) - \Grad\phi \big)} 
\\&
+\frac{h}{2}\sumi \sumK\intK{r_h \Lapmi(\Pim\phi) \Ov{\abs{\vih}} }
+\frac{h}{2}\sumi \sumK\intK{r_h \Ov{\pdedgei(\Pim\phi)} \pdmeshi{\abs{\vih}}},
\end{aligned}
\end{equation}
\begin{equation}\label{conv2}
\begin{aligned}
 \sumi \intO{ \avsi{\divup[r_h,\vvh]} \PiEi \phi } 
&=\intO{ r_h \vvh \cdot \Grad\phi }  
+\intO{ r_h \vvh \cdot \big(\Gradedge(\Pim \PiEi \phi) - \Grad\phi \big)} 
\\&
+\frac{h}{2}\sumi \sumK\intK{r_h \Lapmi(\Pim\PiEi\phi) \Ov{\abs{\vih}} }\\
&+\frac{h}{2}\sumi \sumK\intK{r_h \Ov{\pdedgei(\Pim\PiEi\phi)} \pdmeshi{\abs{\vih}}}.
\end{aligned}
\end{equation}
\end{lemma}
\begin{proof} First, we use the integration by parts formulae stated in Lemma \ref{lem_ibp} and the equality \eqref{up_avg} to get  
\begin{equation*}
\begin{aligned}
 \sumK \intK{\divup[r_h,\vvh] \Pim \phi }
&= \sumi \intDsi{\Upi[r_h,\vvh] \pdedgei(\Pim \phi)}
\\& = 
\sumi \intDsi { \avsi{r_h} \vih  \pdedgei(\Pim \phi) } 
-\frac{h}{2}\sumi \intDsi { \abs{\vih} (\pdedgei r_h) \pdedgei(\Pim \phi)} 
\\& = 
\intO{ r_h \vvh \cdot \Gradedge(\Pim \phi) } 
+ \underbrace{ \frac{h}{2}\sumi \sumK\intK{r_h \pdmeshi\big( \pdedgei(\Pim\phi)\abs{\vih}\big)} }_{=I}
\\& =
\intO{ r_h \vvh \cdot \Grad\phi }  
+\intO{ r_h \vvh \cdot \big(\Gradedge(\Pim \phi) - \Grad\phi \big)}  + I.
\end{aligned}
\end{equation*}
Further, using the chain rule, the term $I$ can be  written as 
\[  I = \frac{h}{2}\sumi \sumK\intK{r_h \Lapmi(\Pim\phi) \Ov{\abs{\vih}} }
+
\frac{h}{2}\sumi \sumK\intK{r_h \Ov{\pdedgei(\Pim\phi)} \pdmeshi{\abs{\vih}}}
\]
which implies \eqref{conv1}. The proof of \eqref{conv2} is more or less similar, and thus we omit the details. 
\end{proof}

We shall  also need the following Sobolev--Poincar\' e--type inequality which can be proved exactly as in  \cite[Theorem 11.23]{FeiNov_book}. 
\begin{lemma}\label{lem_sobolev}
Let $r_h>0$ be a scalar function satisfying
\[
0< \intO{ r_h } =c_M, \quad \intO{ r_h^\gamma }\leq c_E \quad \mbox{ for  } \gamma >1,
\]
where the positive constants $c_M$ and $c_E$ are independent of the mesh parameter $h$. Then the following Sobolev--Poincar\' e--type inequality holds true:
\begin{equation}\label{ineq_sobolev}
\norm{\vvh}_{L^6(\Omega)}^2 \leq c \intO {|\Gradedged \vvh|^2}  + c \left( \intO{ r_h |\vvh| } \right)^2 , 
\end{equation}
for any $\vvh \in \Yspace$, where the constant c depends on $c_M,$ $c_E$ but not on the mesh parameter. 

In particular, by setting $r_h\equiv 1$ we have 
\[
\norm{\vvh}_{L^6(\Omega)}^2 \leq c  \left( \norm{ \Gradedged \vvh}_{L^2(\Omega)}^2  +   \norm{ \vvh }_{L^1(\Omega)}^2  \right).
\]
\end{lemma}
 
Since the convergence proof presented in Section~\ref{sec:convergence} is based on the theory of dissipative measure--valued solutions (DMV), for completeness, we recall the definition of DMV solution \cite[Definition 2.1]{FGSGW} and the related weak--strong uniqueness principle \cite[Theorem 4.1]{FGSGW} for the compressible Navier--Stokes equations. 

\begin{defi}[DMV solution]\label{def_dmvs} 
We say that a parametrized family of probability measures $\{ \Nu _{t,x} \}_{(t,x)\in (0,T)\times\Omega}$,
\begin{equation*}
\Nu _{t,x} \in L^{\infty}_{weak} \Big((0,T)\times\Omega;\, \mathcal{P}(Q) \Big),\  Q = \left\{ [\vr, \vu] \ \Big|\ \vr \in [0, \infty), \ \vu \in  R^d \right\},
\end{equation*}
is \emph{a DMV solution} of the Navier--Stokes system in $(0,T)\times\Omega$ with the initial condition $\Nu_{0,x}\in \mathcal{P}(Q)$ and dissipative defect $\mathcal{D} \in L^{\infty}(0,T),\ \mathcal{D} \ge 0$, if the following holds:
\begin{itemize}
\item
\begin{equation*}
\left[\int_{\Omega} \myangle{\Nu_{t,x};\vr} \phi(t,\cdot)\dx\right]_{t=0}^{t=\tau} =
\int_0^{\tau}  \int_{\Omega}[\myangle{\Nu_{t,x};\vr} \pdt \phi + \myangle{\Nu_{t,x};\vr \bfu}\cdot \Grad  \phi] \dx \dt
\end{equation*}
for any $0 \leq \tau \leq T$, and any $\phi \in C^1\big([0,T]\times \Omega\big)$;

\item
\begin{multline*}
\left[\int_{\Omega} \myangle{\Nu_{t,x};\vr \vu} \bfPhi(t,\cdot)\dx \right]_{t=0}^{t=\tau}  =
\int_0^{\tau}  \int_{\Omega}[\myangle{\Nu_{t,x};\vr \vu} \pdt \bfPhi + \myangle{\Nu_{t,x};\vr \bfu \otimes \vu + p(\vr) \mathbb I} : \Grad  \bfPhi  ] \dx \dt
\\
- \int_0^{\tau} \int_{\Omega} \mathcal{S}(\Grad \vu): \Grad \bfPhi \dx \dt
+ \int_0^{\tau} \myangle{R^M;\Grad  \bfPhi }\dt
\end{multline*}
for any $0 \leq \tau \leq T$, and any $\bfPhi \in C^1_c\big([0,T]\times\Omega; R^d\big)$,
where
\[
 \vu_{t,x} = \left< \Nu_{t,x} ; \vu \right>, \vu \in L^2(0,T; W^{1,2}(\Omega; R^d)), 
\text{ and }
R^M\in L^1\big(0,T;\mathcal{M}({\Omega})\big);
\]

\item
\begin{equation*}
\left[\int_{\Omega} \left\langle\Nu_{t,x};\frac12 \vr \bfu^2+H(\vr)\right\rangle \dx \right]_{t=0}^{t=\tau}  + \int_0^{\tau} \int_{\Omega} \mathcal{S}(\Grad \vu):\Grad \bfu \dx \dt + \mathcal{D}(\tau) \le 0,
\end{equation*}
for a.a. $0 \leq \tau \leq T$.
The dissipation defect $\mathcal{D}$ dominates the concentration measure $R^M$, specifically,
\begin{equation*}
\left|\myangle{R^M(\tau); \phi }\right| \lesssim \xi(\tau) \mathcal{D}(\tau) \norm{\phi }_{C({\Omega})}, \mbox{ for some}\ \xi \in L^{1}(0,T).
 \end{equation*}
\end{itemize}
\end{defi}

\begin{thm}[DMV weak--strong uniqueness]\label{thm_ws_principle}
Let $\Omega \subset R^d,$ $d=2,3,$ be a space--periodic domain. Suppose the pressure $p$ satisfies \eqref{assumption_p}. Let $\mathcal{V}_{t,x}$ be a dissipative measure--valued solution to the barotropic Navier--Stokes system \eqref{strong_ns} in $(0,T) \times \Omega$ with the initial state represented by $\mathcal{V}_0$ in the sense specified in Definition~\ref{def_dmvs}. Let $(\vr, \vu)$ be a strong solution of \eqref{strong_ns} in $(0,T) \times \Omega$ belonging to the class
\begin{equation*}
\vr, \Grad  \vr, \vu, \Grad  \vu \in C([0,T]\times {\Omega}),\ \partial_t \vu \in L^2\left(0,T;C( {\Omega};R^d) \right), \ \vr>0. 
\end{equation*}
Then, if the initial states coincide, meaning
$\mathcal{V}_{0,x} = \delta_{(\vr(0,x),\vu(0,x))}$ for a.a. $x \in \Omega,$ then the dissipation defect $\mathcal{D} = 0,$ and
$\mathcal{V}_{\tau,x} = \delta_{(\vr(\tau,x),\vu(\tau,x))}$ for a.a. $\tau \in (0, T),$ $x \in \Omega.$
\end{thm}
We refer the interested readers to \cite{FGSGW} for further  discussion about DMV solutions to the compressible Navier--Stokes equations.
\begin{remark}
The DMV weak--strong uniqueness result was originally presented for the no-slip boundary conditions. Note that it can be extended for the periodic boundary conditions in a  straightforward manner. 
\end{remark}

\subsection{The numerical scheme}
We are now ready to introduce a novel implicit in time Marker-And-Cell (MAC) finite difference  scheme originally proposed by Ho\v{s}ek and She~\cite{HS_MAC}. The original scheme was based on the set of point values on the centers of the elements and edges. Here we slightly reformulate the scheme such that the discrete problem hold on the whole domain thanks to the piecewise constant extension defined in~\eqref{PCE}. 
\begin{defi}[MAC scheme]
Given the discrete initial values 
\begin{align*}\label{IC}
\left( \varrho^0_h, \avu_h^0 \right)= \left( \Pim \varrho_0, \Pim \vu_0\right)
\end{align*}
we seek the solution $ (\vr_h^n,\vu_h^n ) \in \Xspace \times \Yspace$ satisfying
\begin{subequations}\label{scheme}
\begin{equation}\label{scheme_D}
D_t \vrh^n + \di_{\Up} [\varrho^n_h, \vu_h^n] - h^\alpha \Lapm \vrh^n= 0,
\end{equation}
\begin{multline}\label{scheme_M}
 D_t \avs{ \vrh^n \Ov{\uih^n} }^{(i)} + \avs{\di_\Up[\vrh^n \Ov{\uih^n}, \vu_h^n] }^{(i)}
+ \pdedgei p(\vrh^n)
- \mu \Lape \uih^n - (\mu+\lambda) \pdedgei \dih \vuh^n
\\= h^\alpha  \sumj \avs{ \pdmesh^{(j)} \left( \avs{\auih^n}^{(j)}  (\pdedge^{(j)} \vrh^n)  \right) }^{(i)} ,
\end{multline}
for all $i=1,\ldots, d$, and for all $n = 1,\dots, N_t,$ with the parameter $\alpha$ satisfying 
\begin{equation}\label{Choice_alpha}
\alpha\in (1, 2\gamma- d/3)\; \text{ for } \gamma\in(1,2), \quad \text{and }\;
\alpha>1 \; \text{ for } \gamma \geq 2.
\end{equation}
\end{subequations}
\end{defi}
 \noindent We recall from \cite{HS_MAC} the important properties of the scheme \eqref{scheme}:
\begin{itemize}
\item \textnormal{\bf Existence of solution to \eqref{scheme}.}\\
Let $ (\vr_h^0,\vu_h^0) \in \Xspace \times \Yspace$ be such that $ \vr_h^0 >0 $ (that is $\vr_K^0 >0 $ for any $K \in \mesh$). Then there exists a solution $(\vr_h,\vu_h)=\{(\vr_h^n,\vu_h^n)\}_{n=1}^{N_t} \in \Xspace \times \Yspace $ to the scheme \eqref{scheme}. We refer the readers to \cite[Theorem 3.7]{HS_MAC} for the proof. 

    \item \textnormal{\bf Discrete conservation of mass.} \\
    Summing \eqref{scheme_D} over $K \in \mesh$ immediately yields the conservation of mass, i.e.,
\begin{equation*}
	\int_\Omega \vr^n_h \dx = \int_\Omega \vr^0_h \dx = M_0, \quad n=1,\ldots, N_t.
\end{equation*}
Indeed, it is a simple consequence of \eqref{pp} with $\phi \equiv 1$ and Corollary~\ref{C:intdivup0}.
\item \textnormal{\bf Positivity of discrete density.}\\
Any solution $(\vr_h^n,\vu_h^n) \in \Xspace \times \Yspace $ to \eqref{scheme} satisfies $ \vr^n_h > 0$ provided $ \vr_h^{n-1}>0,$  $n=1,\ldots,N_t.$  \\ See \cite[Lemma 3.2]{HS_MAC} for the  proof. 
\end{itemize}

\subsection{Main results}
The first main result is the convergence to the DMV and strong solutions on the lifespan of the latter. 
\begin{thm}[Convergence]\label{thm_convergence}
Let $\{(\vr_h^n,\vuh^n)\}_{n=1}^{N_T}$ be a family of numerical solutions obtained by the scheme \eqref{scheme} with $\Delta t \approx h$ for all $\gamma >1$.  
Let the initial data $(\vr_0, \bfu_0)$  satisfy \eqref{initial_assumption}. Then, we have the following convergence results:
\begin{itemize}
\item Any Young measure $\{ {\Nu_{t,x}}\}_{t,x\in(0,T)\times\Omega} $ generated by $(\vr_h^n,\vuh^n)$ for $h\rightarrow 0$ represents a DMV solution of the Navier--Stokes system (\ref{strong_ns}) in the sense of Definition~\ref{def_dmvs}. 
\item In addition, suppose that the Navier--Stokes system (\ref{strong_ns}) endowed with the initial data $(\vr_0, \bfu_0)$ and periodic boundary conditions admits a regular solution $(\vr,\bfu)$ belonging to the class 
\begin{equation*}
\vr, \Grad  \vr, \bfu, \Grad  \bfu \in C([0,T]\times {\Omega}),\ \partial_t \vu \in L^2\left(0,T;C( {\Omega};R^d) \right),\ \vr>0. 
\end{equation*}
Then
\begin{equation*}
\vr_h \rightarrow \vr \mbox{ (strongly) in } L^{\gamma}\left((0,T) \times \Omega \right),\
\vu_h \rightarrow \vu \mbox{ (strongly) in } L^2\left((0,T)\times \Omega; R^d \right).
\end{equation*}
\end{itemize}
\end{thm}

Note that the existence of the strong solution has been reported in \cite{Cho}. Further, assuming the existence of a more regular strong solution and ``large" values of $\gamma >\frac{d}{2}$, we deduce the following convergence rate.  
\begin{thm}[Convergence rate]\label{thm_error_estimates}
Let $\gamma \geq \frac{d}2$. Let $(r,\vU)$ be a strong solution to the Navier--Stokes system \eqref{strong_ns} which belongs to the class
\begin{equation}\label{ST_class}
r\in C^2([0,T]\times \Omega),\ \underline{r} \leq  r(t,x) \leq \Ov{r}, \quad \vU \in C^2([0,T]\times \Omega;R^d).
\end{equation}
Then under the assumptions of Theorem~\ref{thm_convergence} there exists a positive number 
$$c=c(M_0, E_0, \underline{r}, \overline{r}, |p'|_{C^1([\underline{r},\overline{r}])}, \|(\Grad r, \pd_t r, \pd_t \Grad r, \pd_t^2 r, \vU, \Grad \vU, \Grad^2 \vU, \pd_t\vU, \pd_t \Grad \vU)\|_{L^\infty((0,T)\times \Omega)})$$
depending tacitly also on $T,$  $\gamma,$ \textnormal{diam}$(\Omega)$, $|\Omega|$, such that there holds 
\begin{equation*}
\begin{aligned}
&\sup_{0\leq n \leq N}\mathfrak{E}(\vr_h^n,\vu_h^n | r_h^n, \vUh^n) 
+ \TS \sum_{n=1}^N \frac{\mu}{2}\intO{|\Gradedged(\vu_h^n-\vU_h^n)|^2 }
+ \TS \sum_{n=1}^N (\mu+\lambda) \intO{|\dih(\vu_h^n-\vU_h^n)|^2 }
 \\
&\leq c \left( h^A + \sqrt{\TS} + \mathfrak{E}(\vr_h^0, \vu_h^0 | r(0), \vU(0))\right) ,
\end{aligned}
\end{equation*}
where the convergence rate reads 
\begin{equation}\label{A}
A=\min \left\{ \frac{2\gamma-d}{\gamma},  \frac{1}{2} \right\}.
\end{equation}
\end{thm}

\begin{remark}
Our Theorem~\ref{thm_error_estimates} states the same convergence rate as \cite[Theorem 3.2]{GallouetMAC}. However, we would like to point out that the reference \cite{GallouetMAC} requires an assumption on the asymptotic behaviour of the pressure while we do not need it due to the additional artificial diffusion.   
\end{remark}

\section{Proof of Theorem \ref{thm_convergence}: convergence}\label{sec:convergence}

 The strategy of employing the DMV solutions as a tool for the convergence analysis of a numerical scheme consists of two steps:
 \begin{itemize}[noitemsep,topsep=0pt,leftmargin=25pt]
\item[i)] showing that a sequence of approximate solutions generates a DMV solution 
\item[ii)]
proving convergence to  strong solution via the DMV weak--strong uniqueness principle.
\end{itemize}
Thanks to the DMV weak--strong uniqueness result derived in \cite[Theorem 4.1]{FGSGW} (see also Theorem \ref{thm_ws_principle}), for the proof of convergence of  numerical solutions towards the strong solution it suffices to show that a sequence of solutions to the proposed MAC scheme \eqref{scheme} generates a DMV solution in the sense of Definition~\ref{def_dmvs}. 
To this end we shall prove the essential properties:
energy stability and consistency of the scheme. We recall some of the necessary estimates from \cite{HS_MAC}, where the stability estimates and the consistency formulation of the MAC scheme \eqref{scheme} in the case of the no-slip boundary condition  were derived for the adiabatic coefficient $\gamma \in(1,2)$. 
Note that the space--periodic  setting studied in the present paper causes no major difference in the proof. 
 The main difference lies in applying the Sobolev--Poincar\' e--type inequality to bound the discrete velocity in $L^2(0,T; L^6(\Omega))$, see~\eqref{ineq_sobolev}. A second difference is to complement the proof also for $\gamma \geq 2$. 

\subsection{Energy stability}
The essential feature of any numerical scheme is its stability. We now recall the energy inequality derived for the scheme \eqref{scheme} in the recent work  of Ho\v sek and She \cite{HS_MAC}.
\begin{lemma}\label{thm_stability}(\cite[Theorem 3.5]{HS_MAC})
Let  $p$ satisfy the pressure law \eqref{assumption_p}, and let $(\vrh, \vuh)$ be a numerical solution obtained by the scheme \eqref{scheme}. Then, for all $m = 1, \dots, N_t,$ it holds that
\begin{equation}\label{energy_estimate}
\intO{ \left( \frac12 \vrh^m  \abs{\Ov\vu^m_h}^2 + \Hc(\vrh^m) \right) }  + \mu \Delta t \norm{\Gradedged \vu_h^m}_{L^2(\Omega)} + (\mu +\lambda) \Delta t \norm{\dih \vu_h^m}_{L^2(\Omega)} 
 + \sum_{j=1}^4 N_j^{m}  
\leq   E_0 ,
\end{equation}
where  
$\displaystyle E_0=\intO{\Big(\frac12 \vrh^0 |\Ov\vu^0_h|^2 + \Hc(\vrh^0)  \Big)}$ and $\displaystyle  N_j^{m} \geq 0$ with
\begin{align*}
 N_1^m &=  \Delta t \sum_{n=1}^m \sumi \sum_{\sigma \in \edgesi} \int_{D_\sigma} \bigg( (h^\alpha + h|u^n_{i,\sigma}|)  \Hc''(\vr^n_{h,\dagger}) \abs{\pdedgei \vrh^n}^2 \bigg) \dx , \\
 N_2^m &=  (\Delta t)^2 \sum_{n=1}^m \intO{ \frac{\Hc''(\vrh^{n-1,n})}{2} |\pdt \vrh^n|^2 } ,  \\
 N_3^m &= (\Delta t)^2  \sum_{n=1}^m \intO{\frac{\vr^{n-1}_h}{2}  \abs{D_t \Ov{\vu_h^n}}^2} , \\
 N_4^m &= \frac{1}{4} \TS\, h  \sum_{n=1}^m \sumfaceint \int_{D_\sigma} \vrh^{n,up}|\udn|\; \abs{\Gradedge \Ov{\vuh^n}}^2 \dx.
\end{align*}
Here $\vrh^{n-1,n} \in \co{\vrh^{n-1}, \vrh^{n}}$, 
$\vr^n_{h,\dagger} \in  \co{\vr^{n}_K, \vr^{n}_L}$ for any $\sigma=K|L$ are the remainder terms from the Taylor expansions. 
\end{lemma}

\subsubsection{Uniform bounds}
The total energy inequality \eqref{energy_estimate} implies the following \emph{a priori} estimates.
\begin{coro}[Uniform bounds]\label{ests}
Let $(\vrh, \vuh)$ be the solution to the scheme \eqref{scheme} with the pressure satisfying \eqref{assumption_p}. 
Then there exists $c>0$ dependent on the initial mass $M_0$ and energy $E_0$ but independent of the parameters $h$ and $\TS$ such that
\begin{subequations}\label{ests1}
\begin{align} 
&
\norm{\vrh \avu_h^2}_{L^{\infty}L^1}  \aleq 1, \quad 
\norm{\vrh}_{L^{\infty}L^\gamma} \aleq 1, \quad 
 \norm{\vrh\avu_h}_{L^\infty L^{\frac{2\gamma}{\gamma+1}}}  \aleq 1,
\label{est_ener}
\\&
\norm{\dih \vu_h}_{L^2L^2} \aleq 1, \quad 
\norm{\Gradedged \vuh}_{L^2 L^2}  \aleq 1,\quad  
\norm{\vuh}_{L^{2}L^6}  \aleq 1, \quad
\label{est_u}
\\&
\Delta t  \int_0^T \intO{\vrh(t-\TS) |\pdt \Ov{\vu_h}|^2} \dt \aleq 1, 
\label{est_dtu}
\\&
 h \int_0^T \sumfaceint \int_{D_\sigma} \vrh^{up}|\vuh \cdot \vn |\; \abs{\Gradedge \Ov{\vuh}}^2 \dx\dt \aleq 1 , \label{est_r_ujump}
\\&
 \int_0^T   \intO{ \bigg( (h^\alpha + h|u_{\sigma}|)  \Hc''(\vr_{h,\dagger}) \abs{\Gradedge \vrh}^2 \bigg) } \dt   \aleq 1,\label{est_adf}
\end{align}
\end{subequations}
where $\vr_{h,\dagger} \in \co{\vr_K, \vr_L}$ for any $\sigma=K|L \in \edgesint$.
\end{coro}
Further, it is convenient to estimate the following 
norms of the density $\vrh$ and the momentum $\mathbf{m}_h\equiv\vrh\avu_h$. 
\begin{lemma}\label{lem_est_l2l2}
In addition to the assumption of Lemma \ref{thm_stability}, let $h\in(0,1)$. Then there hold
\begin{equation}\label{est_l2l2}
\begin{aligned}
\norm{\vrh}_{L^2 L^2 }  \aleq  h^\beta, \quad 
\beta = 
\begin{cases}
 \max\left\{ - \frac{3\alpha+d}{6\gamma}, \frac{\gamma-2}{2\gamma}d \right\}, & \mbox{if } \gamma \in(1,2), \\
0, &\mbox{if } \gamma \geq 2, 
\end{cases}
\\
\norm{\vrh}_{L^2 L^{6/5} }  \aleq  h^\zeta, \quad
\zeta= \begin{cases}
 \max\left\{- \frac{3\alpha+d}{6\gamma},  \frac{\gamma-2}{2\gamma}d, \frac{5\gamma -6}{6\gamma}d \right\}, & \mbox{if } \gamma \in(1,\frac65), \\
0, &\mbox{if } \gamma \geq \frac{6}{5},
\end{cases}
\\
\norm{\vrh \avu_h} _{L^2L^2} \aleq h^\beta, \quad 
\beta = 
\begin{cases}
{ - \frac{3\alpha+d}{6\gamma}}, & \mbox{ if } \gamma \in(1,2), \\
\frac{\gamma-3}{3\gamma}d, &\mbox{ if } \gamma \in [2,3), \\
0, &\mbox{ if } \gamma \geq 3. 
\end{cases}
\end{aligned}
\end{equation} 
\end{lemma}
\begin{proof} First, for the case $\gamma \geq 2$, it is clear that the first estimate of  \eqref{est_l2l2} holds. Indeed, 
\[ \norm{\vrh}_{L^2L^2} \aleq \norm{\vrh}_{L^\infty L^{\gamma}} \aleq 1.  
\]
Concerning the case $\gamma \in (1,2)$ we show the proof in two steps. 
On one hand a direct application of the inverse estimate \eqref{inv_est} leads to 
\[
\norm{\vrh}_{L^2 L^2 }    \aleq h^{d(\frac{1}{2} -\frac{1}{\gamma})} \norm{\vrh}_{L^\infty L^{\gamma}} \aleq h^{d(\frac{1}{2} -\frac{1}{\gamma})}. \]
On the other hand, we may start by recalling the Sobolev inequality for the broken norm~\cite[Lemma A.1]{Chainais} 
\[ \norm{f_h}_{L^6}^2 \aleq
\norm{f_h}_{L^2}^2 +  \norm{\Gradedge f_h}_{L^2}^2, \;  f_h\in \Xspace,
\]
and the algebraic inequality
\[ a \gamma \left(\vr_L^{\gamma/2} -\vr_K^{\gamma/2} \right)^2  \leq \frac{\pd ^2 \Hc (z) }{\pd \vr^2} (\vr_L -\vr_K)^2,\; \forall \; z \in \co{\vr_L, \vr_K}, \; \vr_L, \vr_K>0 \text{ provided } \gamma \in(1,2).
\]
Then the estimate of the density jumps \eqref{est_adf} indicates that
\[  \norm{\Gradedge \vrh^{\gamma/2}}_{L^2L^2}^2 =
\int_0^T  \intO{ |\Gradedge \vrh^{\gamma/2}|^2} \dt
\aleq \int_0^T  \intO{ \Hc''(\vr_{h,\dagger}) |\Gradedge \vrh|^2} \dt
 \aleq h^{-\alpha}.
 \]
Applying the above inequalities together with the inverse estimate \eqref{inv_est} and the estimate \eqref{est_ener} we derive
\begin{equation*}
\begin{aligned}
\norm{\vrh}_{L^1 L^\infty } &
= \int_0^T \norm{\vrh ^{\gamma/2}}_{ L^\infty }^{2/\gamma}  \dt
\leq \int_0^T \left( h^{-d/6} \norm{\vrh ^{\gamma/2}}_{ L^6 } \right) ^{2/\gamma}  \dt \\
&\leq h^{-d/(3\gamma)} \int_0^T \left( \norm{\vrh ^{\gamma/2}}_{ L^2 }^2
+ \norm{\Gradedge \vrh^{\gamma/2}}_{L^2} ^2  \right)^{1/\gamma}  \dt
\leq h^{-d/(3\gamma)}  \left( \norm{\vrh}_{ L^1 L^\gamma } + \norm{\Gradedge \vrh^{\gamma/2}}_{L^{\gamma/2}L^2}  ^{2/\gamma}  \right)\\
&\leq h^{-d/(3\gamma)}  \left( \norm{\vrh}_{ L^\infty L^\gamma } + \norm{\Gradedge \vrh^{\gamma/2}}_{L^2L^2}  ^{2/\gamma}  \right)
\aleq  h ^{ - \frac{3\alpha+d}{3\gamma}}.
\end{aligned}
\end{equation*}
Further application of the above inequality together with the Gagliardo--Nirenberg interpolation inequality, H\"older's inequality, and the density estimate stated in \eqref{est_ener} yields 
\begin{equation*}
 \norm{\vrh}_{L^2 L^2 } = \left( \int_0^T \norm{\vrh}_{L^2}^2 \dt  \right)^{1/2}
 \leq \left( \int_0^T \norm{\vrh}_{L^1}\norm{\vrh}_{L^\infty } \dt  \right) ^{1/2}
   \leq \norm{\vrh}_{L^\infty L^1}^{1/2} \norm{\vrh}_{L^1 L^\infty }^{1/2}
   \aleq  h^{ - \frac{3\alpha+d}{6\gamma}}.
\end{equation*}
Collecting the above results finishes the proof of the first estimate of \eqref{est_l2l2}.

Next, the second estimate of \eqref{est_l2l2} can be shown in the following way. 
First, it is obvious for $\gamma \geq \frac65$ that 
\[ \norm{\vrh}_{L^2L^{6/5}} \aleq \norm{\vrh}_{L^\infty L^\gamma} \aleq 1.
\]
Second, we show the proof for $\gamma \in(1,6/5)$ in two steps. 
On one hand, it is easy to observe that
\[\norm{\vrh}_{L^2L^{6/5}} \aleq \norm{\vrh}_{L^2 L^2} \aleq h^\beta_0, \; \beta_0 = \max\left\{- \frac{3\alpha+d}{6\gamma}, \frac{\gamma -2}{2\gamma}d \right\} .
\]
On the other hand, due to the inverse estimates \eqref{inv_est} we have
\[ \norm{\vrh}_{L^2L^{6/5}} \aleq \norm{\vrh}_{L^\infty L^{6/5}}  \aleq h^{d(\frac{5}{6}- \frac{1}{\gamma})} \norm{\vrh}_{L^\infty L^{\gamma}}
\]
 which completes the proof of the second estimate of \eqref{est_l2l2}. 

The last estimate of \eqref{est_l2l2} can be shown in the following way: if $\gamma \in(1,2)$
\begin{equation*}
 \norm{\vrh \avu_h} _{L^2L^2} \aleq
\norm{ \sqrt{\vrh} }_{L^2 L^\infty} \norm{ \sqrt{\vrh} \avu_h} _{L^\infty L^2}
=\norm{\vrh}_{L^1 L^\infty }^{1/2}  \norm{ \vrh \avu_h^2} _{L^\infty L^1} ^{1/2}
\aleq h^{ - \frac{3\alpha+d}{6\gamma}}.
\end{equation*}
In the case $\gamma \geq 3$, it follows by H\"older's inequality
\[
 \norm{\vrh \avu_h}_{L^2L^2} \aleq  \norm{\vrh}_{L^\infty L^3}  \norm{\vuh}_{L^2 L^6} \aleq \norm{\vrh}_{L^\infty L^\gamma} \norm{\vuh}_{L^2 L^6} \aleq 1.
\]
Finally for $\gamma \in [2,3)$, we have by inverse estimate \eqref{inv_est} and H\"older's inequality that 
\[ \norm{\vrh \avu_h}_{L^2L^2} \aleq  \norm{\vrh}_{L^\infty L^3}  \norm{\vuh}_{L^2 L^6} \aleq h^{d(\frac13 -\frac{1}{\gamma})}\norm{\vrh}_{L^\infty L^\gamma} \norm{\vuh}_{L^2 L^6} 
\aleq h^{d\frac{\gamma-3}{3\gamma}},
\]
which completes the proof. 
\end{proof}

Next we report the dissipation estimates on the density. 
\begin{lemma}(\cite[Lemma 4.3]{GallouetMAC})\label{lem_Ga}
For any $(\vrh, \vuh)$ satisfying the assumptions of Lemma \ref{thm_stability} there holds 
\begin{equation}\label{Ga1}
 \int_0^T  \sum_{ \sigma \in \edgesint } \int_{\sigma =K|L}{  \frac{ ( \vr_K -\vr_L )^2 }{\max\{ \vr_K, \vr_L \}}  \abs{\vuh \cdot \bfn }  } \dS \dt   \leq c
\end{equation}
for $c=c(\gamma, E_0)>0$ provided $\gamma \geq 2$.
\end{lemma}
The following lemma completes the list of useful estimates for the derivation of the consistency formulation.
\begin{lemma}\label{lem_S1}
Under the assumption of Lemma \ref{thm_stability} there exists $c=c(M_0, E_0)>0$  independent of $h$ and $\TS$ such that
\begin{equation}\label{S1}
 \int_0^T  \sum_{ \sigma \in \edgesint } \int_{\sigma =K|L}{\abs{\jump{\vrh}  {\vuh} \cdot\bfn}} \dS \dt   \leq c h^{\beta},
\end{equation}
where 
\[ \beta= \begin{cases}  -\frac12 & \mbox{ if }\gamma \geq \frac65, \\
-\frac12 + \frac{d}{2} \frac{5 \gamma -6}{6\gamma}, & \mbox{ if } \gamma \in (1,\frac65).
\end{cases}
\]
\end{lemma}
\begin{proof}
First, for $\gamma \geq 2$ we apply \eqref{Ga1} and get 
\[
\begin{aligned}
& \int_0^T  \sum_{ \sigma \in \edgesint } \int_{\sigma =K|L}{ \abs{ \jump{\vrh}   \vuh \cdot \bfn }  } \dS \dt   
\\ &  \leq  
\left( \int_0^T  \sum_{ \sigma \in \edgesint } \int_{\sigma =K|L}{ \frac{ \jump{  \vrh } ^2 }{\max\{\vr_L, \vr_K \} }  \abs{\vuh \cdot \bfn } } \dS \dt   \right)^{1/2}
\left( \int_0^T  \sum_{ \sigma \in \edgesint } \int_{\sigma =K|L}{  \abs{\vuh \cdot \bfn } \max\{\vr_L, \vr_K \} } \dS \dt   \right)^{1/2}
\\&
\aleq   h^{-1/2} \norm{\vrh}_{L^2 L^{6/5}}^{1/2} \norm{\vuh}_{L^2 L^6}^{1/2}
\aleq h^{-1/2} 
\end{aligned}
\]
as $\norm{\vrh}_{L^2 L^{6/5}} \aleq \norm{\vrh}_{L^\infty L^\gamma} \leq c(E_0)$ provided $\gamma \geq 2$.


Next, for $\gamma \in (1,2)$ it is easy to check that $\Hc''(\vr_{h,\dagger}) (\vr_{h,\dagger}+1) \geq 1$. Thus we derive 
\[
\begin{aligned}
& \int_0^T  \sum_{ \sigma \in \edgesint } \int_{\sigma =K|L}{ \abs{\jump{\vrh}  \vuh \cdot \bfn }  } \dS \dt   
 \leq  
 \int_0^T  \sum_{ \sigma \in \edgesint } \int_{\sigma =K|L}{ \abs{\jump{\vrh}  \vuh \cdot \bfn } \underbrace{\sqrt{\Hc''(\vr_{h,\dagger}) (\vr_{h,\dagger}+1)} }_{\geq 1} } \dS \dt   
\\ &  \leq 
\left(\int_0^T  \sum_{ \sigma \in \edgesint } \int_{\sigma =K|L}{ \Hc''(\vr_{h,\dagger}) \jump{\vrh}^2  \abs{\vuh \cdot \bfn }  } \dS \dt 
\right) ^{1/2}
\left(\int_0^T  \sum_{ \sigma \in \edgesint } \int_{\sigma =K|L}{ (\vr_{h,\dagger}+1) \abs{ \vuh \cdot \bfn }  } \dS \dt 
\right) ^{1/2}
\\&
\leq c   h^{-1/2} \left( \norm{\vrh}_{L^2 L^{6/5}} + 1 \right)^{1/2} \norm{\vuh}_{L^2L^6}^{1/2}
\leq c h^{-1/2} \norm{\vrh}_{L^2 L^{6/5}} ^{1/2}
\leq c h^\beta,
\end{aligned}
\]
where thanks to the second estimate of Lemma~\ref{lem_est_l2l2} $\beta$ reads 
\[ \beta= \begin{cases}  -\frac12, & \mbox{ if }\gamma \in [\frac65,2), \\
-\frac12 + \frac{d}{2} \frac{5 \gamma -6}{6\gamma}, & \mbox{ if } \gamma \in (1,\frac65).
\end{cases}
\]

\end{proof}

\subsection{Consistency formulation}\label{subsec:con}
Another step towards the convergence of a sequence of approximate solutions is the consistency of the numerical scheme. 
\begin{lemma}\label{thm_consistency}
Let the pressure $p$ satisfies \eqref{assumption_p} with $\gamma>1$. 
Let $(\vrh, \vuh)$ be a solution of the numerical scheme \eqref{scheme} with $\Delta t \approx h$.  Then, for any  $\phi \in C_c^2([0,T]\times\Omega),$ and any $\bfPhi \in C_c^2([0,T]\times \Omega;  R^d)$  it holds that
\begin{subequations}\label{cons_form}
\begin{equation} \label{cons1}
- \intO{ \vrh^0 \phi(0,\cdot) }  =
 \int_0^T \intO{ \left[ \vrh \partial_t \phi + \vrh \vuh \cdot \Grad \phi \right]} \dt  + \mathcal{O}( h^{\beta_1}),\  \beta_1>0;
\end{equation}
\begin{multline} \label{cons2}
- \intO{ \vrh^0 \vuh^0 \bfPhi(0,\cdot) }  =
\int_0^T \intO{ \left[ \vrh \avu_h \cdot \partial_t \bfPhi + \vrh \avu_h \otimes \avu_h  : \Grad \bfPhi  + p_h \dix \bfPhi \right]} \dt
\\
 -  \mu \int_0^T \intO{  \Gradedged \vuh : \Grad \bfPhi}  \dt
 -  (\mu+\lambda) \int_0^T \intO{  \dih \vuh : \dix \bfPhi}  \dt
 +  \mathcal{O} (h^{\beta_2}), \ \beta_2>0. 
\end{multline}
\end{subequations}
\end{lemma}
\begin{proof}
First we show the proof for $\gamma \geq2$ in two steps.
\paragraph{Step 1 -- Consistency of density equation.} To show the consistency formulation \eqref{cons1} we  multiply  the discrete density equation \eqref{scheme_D} with $\Pim \phi$ for $\phi \in C_c^2([0,T]\times\Omega)$ and integrate over $\Omega$. In what follows we handle each term of the product separately.
\begin{itemize}
\item \textbf{Time derivative term.} 
It is easy to calculate   
\begin{align*}
& \sum_{n=1}^N \TS \sumK \intK{  D_t\vrh^n\;  \Pim \phi  } = \sum_{n=1}^N \TS \sumK  D_t\vrh^n \intK{ \phi } =\int_0^T \intO{ D_t \vrh(t) \phi(t) }\dt \\
 &= \int_0^T \intO{ \frac{ \vrh(t) -\vrh(t-\TS)}{\TS} \phi(t)} \dt 
\\ &=\int_0^T \intO{ \frac{\vrh(t) \phi(t)}{\TS}} \dt   -\int_{-\TS}^{T-\TS} \intO{ \frac{\vrh(t) \phi(t+\TS)}{\TS}} \dt 
\\ &
= \int_0^T \intO{\vrh(t) \frac{ \phi(t) -\phi(t+\TS)}{\TS}  }\dt  
-  \int_{-\TS}^0 \intO{ \frac{ \vrh(t) \phi(t+\TS)}{\TS}  }\dt +\int_{T-\TS}^T \intO{ \frac{ \vrh(t) \phi(t+\TS)}{\TS}} \dt 
\\& 
= -\int_0^T \intO{\vrh(t) \left(\pd_t \phi(t) +\frac{\TS}{2} \pd_{tt} \phi(t^*)\right) }\dt
 - \intO{\vrh^0 \phi(0)} - \int_0^{\Delta t}\intO{\vrh^0\frac{ \left(\phi(t)-\phi(0)\right)}{\Delta t}}\dt
\\& =-\int_0^T \intO{\vrh(t) \pd_t \phi(t) }\dt - \intO{\vrh^0 \phi(0)}  +I_0
\end{align*}
where $I_0=  -\int_0^T \intO{\vrh(t) \frac{\TS}{2} \pd_{tt} \phi(t^*)  }\dt - \int_0^{\Delta t}\intO{\vrh^0 \pd_t \phi(t^\dagger)}\dt $ for suitable $t^*, t^\dagger \in(t,t+\TS)$. 
Obviously, $I_0$ can be controlled by 
\[\abs{I_0}  \aleq  \TS  \norm{\vrh}_{L^1L^1} \norm{ \phi }_{C^2} +  \TS \norm{\vrh^0}_{L^1} \norm{ \phi }_{C^1} \aleq c\left(M_0, E_0, \norm{ \phi }_{C^2}\right)  \TS .\]
Therefore, we have 
\begin{equation*}\label{CC1}
\sum_{n=1}^N \TS \sumK \intK{ D_t\vrh^n \; \Pim \phi  } + \int_0^T \intO{\vrh \pd_t \phi }\dt  + \intO{\vr_0 \phi(0)} \leq  c\left(M_0, E_0, \norm{ \phi }_{C^2}\right) \TS . \\
\end{equation*}

\item \textbf{Convective term.} Setting $r_h= \vrh$ in \eqref{conv1} for the convective term, we get 
\[
\sum_{n=1}^N \TS \sumK \divup (\vrh,\vuh)\left(|K| \Pim \phi\right) \dx
 =\inttime \intO{ \vrh \vuh \cdot \Grad\phi } \dt + I_1+ I_2+ I_3 ,
\]
where 
\[
\begin{aligned}  
& I_1 = 
\inttime  \intO{ \vrh \vuh \cdot \big(\Gradedge(\Pim \phi) - \Grad\phi \big)} \dt,
\\& I_2 =
\inttime  \frac{h}{2}\sumi \sumK\intK{\vrh \Lapmi(\Pim\phi) \Ov{\abs{\uih}} }\dt,
\\& I_3 =
\inttime \frac{h}{2}\sumi \sumK\intK{\vrh \Ov{\pdedgei(\Pim\phi)} \pdmeshi{\abs{\uih}}} \dt.
\end{aligned}
\]
The terms $I_i$, $i=1,2,3$ can be controlled as follows
\begin{align*}
&\abs{I_1} \aleq h\norm{\vrh}_{L^2 L^{6/5}} \norm{\vuh}_{L^2L^6} \norm{\phi}_{C^2} \aleq c(E_0,\norm{\phi}_{C^2} )  h,
\\&
\abs{I_2} \aleq h\norm{\vrh}_{L^2 L^{6/5}} \norm{\vuh}_{L^2L^6} \norm{\phi}_{C^2} \aleq c(E_0,\norm{\phi}_{C^2} )  h,
\\&
\abs{I_3} \aleq h \norm{\vrh}_{L^2L^2} \norm{\dih \vuh}_{L^2L^2} \norm{\phi}_{C^2} \aleq c(E_0,\norm{\phi}_{C^2} )  h,
\end{align*}
where we have used H\"older's inequality, Lemma~\ref{NC} and the uniform bounds \eqref{ests1}.

\item \textbf{Artificial diffusion term.} Using the integration by parts formula \eqref{pp} twice together with the estimate \eqref{NC3} yields
\[
\begin{aligned}
\sum_{n=1}^N \TS  \sumK \ha \Lapm \vrh^n \left(|K| \Pim \phi\right) = \ha \int_0^T \intO{  \Lapm\vrh  \phi } \dt 
&= \ha \int_0^T \intO{  \vrh \Lapm \phi } \dt
\\ & \aleq  \ha\norm{\vrh}_{L^1L^1} \norm{\phi}_{C^2} 
\leq c(E_0,\norm{\phi}_{C^2} ) \ha.
 \end{aligned}
 \]
\end{itemize}
Collecting the above proves \eqref{cons1} for $\gamma \geq 2$. 

\paragraph{Step 2 -- Consistency of momentum equation.} To show the consistency formulation \eqref{cons2}, we multiply the discrete momentum equation \eqref{scheme_M} with $\Pie \bfPhi$ for $\bfPhi=(\Phi_1,\ldots, \Phi_d) \in C_c^2([0,T]\times\Omega;R^d)$ and integrate over $\Omega$. Then we proceed analogously as in Step 1.
\begin{itemize}
\item \textbf{Time derivative term.} 
Similarly as in Step 1, we have 
\[
\begin{aligned}
& \sum_{n=1}^N \TS \sumi \intDsi{ \avs{D_t (\vrh \auih)}^{(i)} \PiEi \Phi_i }
= \int_0^T  \sumK \int_K{ D_t(\vrh \avu_h) \cdot \bfPhi  } \dx \dt
+I_0
\\& =  \int_0^T \intO{\vrh \avu_h \cdot \pd_t \bfPhi }\dt  - \intO{\vrh^0  \avu_h^0 \cdot \bfPhi(0)}  
 \end{aligned}
 \]
where 
\[I_0=  -\frac{\TS}{2} \int_0^T \intO{\vrh(t)\avu_h \cdot \pd_{tt} \bfPhi(t^*)  }\dt - \int_0^{\Delta t}\intO{\vrh^0\avu_h^0 \cdot \pd_t \bfPhi(t^\dagger)}\dt \]
 for suitable $t^*, t^\dagger \in(t,t+\TS)$. 
Obviously 
\[\abs{I_0}  \aleq  \TS  \norm{\vrh\avu_h}_{L^1L^1} \norm{ \bfPhi }_{C^2} +  \TS \norm{\vrh^0\avu_h^0}_{L^1} \norm{ \bfPhi }_{C^1} \aleq c\left(E_0,\norm{ \vr_0 }_{L^2},\norm{ \vu_0 }_{L^2}, \norm{ \bfPhi }_{C^2}\right)  \TS .\]
Consequently, we have 
\begin{multline*}  
  \sum_{n=1}^N \TS \sumi \intDsi{ \avs{D_t (\vrh \auih)}^{(i)} \PiEi \Phi_i }
   + \int_0^T \intO{\vrh  \avu_h \cdot \pd_t \bfPhi }\dt  + \intO{\vr_0 \avu_h^0 \cdot \bfPhi(0)}    \\
  \aleq c\left(E_0, \norm{ \bfPhi }_{C^2}\right) \TS .
\end{multline*}
\item \textbf{Convective term.} Setting $(r_h, \phi)=(\vrh \aujh, \Phi_j)$ in \eqref{conv2} for the convective term, we get 

\begin{equation*}
\begin{aligned}
& \sumj \inttime \intO{ \avsi{\divup[\vrh \aujh,\vuh]} \PiEi \Phi_j } \dt
=\inttime \intO{ (r_h \vuh \otimes \vuh) : \Grad\bfPhi }   \dt + I_1+ I_2+ I_3 ,
\end{aligned}
\end{equation*}
where 
\[
\begin{aligned}  
& I_1 = 
 \sumj \inttime  \intO{ \vrh \aujh \vuh \cdot \big(\Gradedge(\Pim \PiEi \Phi_j) - \Grad \Phi_j \big)} \dt,
\\& I_2 =
\sumj \inttime  \frac{h}{2}  \sumi \sumK\intK{\vrh\aujh \Lapmi(\Pim\PiEi \Phi_j) \Ov{\abs{\uih}} }\dt,
\\& I_3 =
\sumj \inttime \frac{h}{2}\sumi \sumK\intK{\vrh\aujh \Ov{\pdedgei(\Pim \PiEi \Phi_j)} \pdmeshi{\abs{\uih}}} \dt.
\end{aligned}
\]
Employing Lemma \ref{lem_est_l2l2} with $\gamma \geq 2$, the terms $I_i$, $i=1,2,3,$ can be controlled as follows 
\begin{align*}
&\abs{I_1} \aleq h \norm{\vrh \avu_h^2}_{L^1L^1} \norm{\phi}_{C^2} \aleq c(E_0,\norm{\phi}_{C^2} ) h,
\\&
\abs{I_2} \aleq h\norm{\vrh \avu_h}_{L^2 L^2} \norm{\vuh}_{L^2L^6} \norm{\phi}_{C^2} \aleq c(E_0,\norm{\phi}_{C^2} ) h,
\\&
\abs{I_3} \aleq h \norm{\vrh \avu_h}_{L^2L^2} \norm{\dih \vuh}_{L^2L^2} \norm{\phi}_{C^2}\aleq c(E_0,\norm{\phi}_{C^2} ) h,
\end{align*}
where we have used H\"older's inequality and the uniform bounds \eqref{ests1}.

\item \textbf{Pressure term.} By \eqref{div_equiv} and the integration by parts formula \eqref{grad_div}, we have 
\begin{equation*}
\begin{aligned}
  \sum_{n=1}^N \TS \sumi  \intDsi{ \pdedgei p(\vrh^n)\;  \PiEi \Phi_i } &= - \int_0^T  \sumK \intK{ p(\vrh) \dih \PiE \bfPhi} \dt
\\ & = - \int_0^T \intO{ p(\vrh )\; \dix \bfPhi} \dt.
\end{aligned}
\end{equation*}

\item \textbf{Diffusion term.} 
Analogously as the pressure term, we have 
\begin{equation*}
 - \sum_{n=1}^N \TS \sumi  \intDsi{ \pdedgei \dih \vuh^n \;  \PiEi \Phi_i } 
 = + \int_0^T \intO{ \dih \vuh \; \dix \bfPhi} \dt.
\end{equation*}

Next, employing the integration by parts formula \eqref{grad_div} we can write
\begin{equation*} 
\begin{aligned}
 -   \sum_{n=1}^N \TS \sumi  \intDsi{ \Lape \uih^n \PiEi \Phi_i }  &=   \int_0^T \intO{ \Gradedged \vuh : \Gradedged \Pie \bfPhi}
\\ &
 =  \int_0^T \intO{ \Gradedged \vuh : \Grad \bfPhi } +  \underbrace{ \int_0^T \intO{ \Gradedged \vuh : (\Gradedged \Pie \bfPhi - \Grad \bfPhi)}}_{=R},
\end{aligned}
\end{equation*}
where  H\"older's inequality and the estimate \eqref{NC2} imply 
\[  \abs{R}  \aleq \norm{\Gradedged \vuh}_{L^2L^2}  \norm{\Gradedged \Pie \bfPhi - \Grad \bfPhi}_{L^2L^2}  
\aleq  c(E_0,  \norm{ \bfPhi}_{C^2}) \; h.   
\]

\item \textbf{Artificial diffusion term.} We apply \eqref{grad_div}, Lemma~\ref{NC} and \ref{lem_est_l2l2} and chain rule to get
\begin{equation*}
\begin{aligned}
 &  h^\alpha \sum_{n=1}^N \TS \sumi  \intDsi{    \sumj \avs{ \pdmesh^{(j)} \left( \avs{\auih^n}^{(j)}  (\pdedge^{(j)} \vrh^n)  \right) }^{(i)}  \PiEi \Phi_i  }
 \\ & = 
  h^\alpha \int_0^T \sumi  \sumj  \intO{     \pdmesh^{(j)} \left( \avs{\auih^n}^{(j)}  (\pdedge^{(j)} \vrh)  \right)   \Ov{ \PiEi \Phi_i}  } \dt
  \\ &=
- h^\alpha \int_0^T \sumi  \sumj  \intO{      \avs{\auih^n}^{(j)}  (\pdedge^{(j)} \vrh)     \pdedgej  \Ov{ \PiEi \Phi_i}  } \dt
  \\ &=
h^\alpha \int_0^T \sumi  \sumj  \intO{    \vrh \pdmeshj \left(  \avs{\auih^n}^{(j)}       \pdedgej   \Ov{ \PiEi \Phi_i}    \right) } \dt
  \\ &=
h^\alpha \int_0^T \sumi  \sumj  \intO{    \vrh \pdmeshj \left(  \avs{\auih^n}^{(j)}   \right) \Ov{    \pdedgej   \Ov{ \PiEi \Phi_i}  }  } \dt
  \\ & \qquad + 
  h^\alpha \int_0^T \sumi  \sumj  \intO{    \vrh \Ov{\avs{\auih^n}^{(j)} } \pdmeshj \left(       \pdedgej   \Ov{ \PiEi \Phi_i}    \right) } \dt
 \\ & \aleq \ha \norm{\vrh}_{L^2L^2}  \left( \norm{\Gradedged \vuh}_{L^2L^2} \norm{\bfPhi}_{C^1}   +    \norm{ \vuh}_{L^2L^2} \norm{\bfPhi}_{C^2}    \right)
\\ & \aleq c(E_0,   \norm{\bfPhi}_{C^2}    ) \ha.
\end{aligned}
\end{equation*}
\end{itemize}
Collecting the estimates of Step 2 proves \eqref{cons2} and finishes the whole consistency proof for $\gamma \geq 2$. Concerning $\gamma \in (1,2)$, the consistency of the numerical scheme \eqref{scheme} with the no-slip boundary conditions was shown in \cite[Theorems 4.6 and 4.7]{HS_MAC}. Note that the proof remains the same for the periodic boundary conditions, thus we omit it here.  
\end{proof}

\subsection{Convergence to DMV solution}
We aim to pass to the limit with the discretization parameter $h \rightarrow 0$ to show the convergence of  a sequence of numerical solutions to the DMV solution. Having established the two essential prerequisites, i.e. stability estimates and consistency formulation in  Lemmas~\ref{thm_stability} and \ref{thm_consistency}, respectively,  the convergence proof can be done analogously as in our recent work \cite{FLMS_FVNS} or the pioneering paper \cite{FL}, in which the same strategy was used. For completeness we briefly recall the main steps.
\paragraph{Weak limit.}  First, the energy estimates \eqref{energy_estimate} yield, at least for suitable subsequences (not relabelled), 
\begin{align*}
\vr_h &\to \vr \ \mbox{weakly-(*) in}\ L^\infty(0,T; L^\gamma(\Omega)),\ \vr \geq 0,\\
 \avu_h, \vuh &\to \bfu \ \mbox{weakly in}\ L^2((0,T) \times \Omega; R^d),
\mbox{where} \ \bfu \in L^2(0,T; W^{1,2}(\Omega)),\\
 \Gradedged \vuh &\to \Grad \bfu \ \mbox{weakly in} \ L^2((0,T) \times \Omega;
R^{d \times d}), \\
 \dih \vuh &\to \dix \bfu \ \mbox{weakly in} \ L^2((0,T) \times \Omega), \\
\vr_h \avu_h  &\to \widetilde{\vr \bfu} \ \  \mbox{weakly-(*) in}\ L^\infty(0,T; L^{\frac{2\gamma}{\gamma + 1}}(\Omega; R^d)),\\
\vr_h (\avu_h \otimes  \avu_h) + p(\vr_h) \bI  & \to
\left\{ \vr \bfu \otimes \bfu + p(\vr) \bI \right\}\ \mbox{weakly-(*) in}\ \left[ L^\infty(0,T; \mathcal{M}(\Omega) )\right]^{d \times d},
\end{align*}
where    $\widetilde{\cdot}$ and $\{\cdot\}$ denote the $L^1$-weak limit and the $ L^\infty(\mathcal{M}(\Omega))$-weak-(*) limit, respectively.
\paragraph{Young measure generated by numerical solutions.}
According to the weak convergence statement, we can conclude that the family of numerical solutions $(\vrh,\vuh)=\{(\vr_h^n, \vuh^n)\}_{n=1}^{N_t}$ generates a Young measure 
\[
\Nu_{t,x} \in L^\infty((0,T) \times \Omega; \mathcal{P}([0, \infty) \times R^d)) \ \mbox{for a.a.}\ (t,x) \in (0,T) \times \Omega, \mbox{ with }\Nu_{0,x} = \delta_{[\vr_0(x), \bfu_0(x)]}
\]
such that
\[
\left< \Nu_{t,x}, g(\vr, \bfu) \right> = \widetilde{g(\vr, \bfu)}(t,x)\ \mbox{for a.a.}\ (t,x) \in (0,T) \times \Omega,
\]
whenever $g \in C([0, \infty) \times R^d)$, and
\[
g(\vr_h, \vuh) \to \widetilde{g(\vr, \bfu)} \ \mbox{weakly in}\ L^1((0,T) \times \Omega).
\]
We refer the reader to, e.g., \cite{BALL2,PED1} for more details on a parametrized measure $\mathcal{V}_{t,x}.$
\paragraph{Passing to the limit.}
We pass to the limit  with $h \to 0$ in the consistency formulation \eqref{cons_form} and the energy inequality \eqref{energy_estimate} to get
\begin{equation} \label{M4}
\left[ \intO{ \left< \Nu_{t,x} ; \vr  \right> \phi (\tau, \cdot) } \right]_{t = 0}^{t = \tau} = \int_0^\tau
\intO{ \left[   \left< \Nu_{t,x} ; \vr  \right> \partial_t \phi + \left< \Nu_{t,x}, \vr \bfu \right>  \cdot \nabla _x \phi \right] } \dt
\end{equation}
for any $0 \leq \tau \leq T$ and any $\phi \in C^\infty([0,T] \times \Omega)$;  \newline
\begin{equation} \label{M5}
\begin{split}
\left[ \intO{ \left< \Nu_{t,x} ; \vr \bfu \right> \cdot  \bfPhi(0, \cdot) } \right]_{t = 0}^{t = \tau} = \int_0^\tau \intO{
\Big[ \left< \Nu_{t,x}; \vr \bfu \right> \cdot \partial_t \bfPhi + \left< \Nu_{t,x}; \vr \bfu \otimes \bfu +p(\vr) \bI  \right>
: \Grad \bfPhi  \Big] } \ \dt\\
- \int_0^\tau \intO{  \mathcal{S}(\Grad \vu): \Grad \bfPhi } \dt
+ \int_0^\tau \intO{ \mathcal{R} : \nabla _x \bfPhi} \dt
\end{split}
\end{equation}
for any $0 \leq \tau \leq T$, $\bfPhi \in \DC([0,T] \times \Omega; R^d)$, where $\mathcal{R}$ is the \emph{concentration remainder},
\[
\mathcal{R} = \left\{ \vr \bfu \otimes \bfu + p(\vr) \bI \right\} - \left< \Nu_{t,x}; \vr \bfu \otimes \bfu + p(\vr) \bI \right>
\in [ L^\infty(0,T; \mathcal{M}(\Omega)) ]^{d \times d};
\]
\begin{equation} \label{M6}
\begin{split}
\left[ \intO{ \frac{1}{2} \left< \Nu_{t,x}; \vr \bfu^2 + \mathcal  \Hc(\vr) \right>  } \right]_{t = 0}^{t = \tau} 
+ \int_0^{\tau} \int_{\Omega} \mathcal{S}(\Grad \vu): \Grad \bfPhi \dx \dt
  + \mathcal{D}(\tau) &\leq 0
\end{split}
\end{equation}
for a.e. $\tau \in [0,T]$, where $\mathcal{D}$ is the \emph{dissipation defect}  
\[
\begin{aligned}
\mathcal{D}(\tau) =& \lim_{h \to 0} \intO{ \left(\frac 1 2 \vrh | {\vuh } |^2 + \Hc(\vrh)\right) } -  \intO{\left< \Nu_{\tau,x};  \frac 1 2 \vr | \vu  |^2 + \Hc(\vr)\right> }
\\&+ \lim_{h \to 0}  \mu \int_0^\tau \intO{ |\Gradedged \vuh|^2 }\ \dt - \mu  \int_0^\tau \intO{ |\Grad \bfu |^2 }\ \dt
\\&+ \lim_{h \to 0} (\mu+\lambda)\int_0^\tau  \intO{ |\dih \vuh|^2 }\ \dt - (\mu+\lambda) \int_0^\tau \intO{ |\dix \bfu |^2 }\ \dt,
\end{aligned}
\]
which, using  \cite[Lemma 2.1]{FGSGW}, can be shown to satisfy  
\begin{equation} \label{M7}
\int_0^\tau \| \mathcal{R} \|_{\mathcal{M}(\Omega)}\ \dt \aleq  \mathcal{D}(\tau) .
\end{equation}
The detailed passage to the limit can be found in our recent work \cite{FLMS_FVNS}. Based on relations \eqref{M4}--\eqref{M7} we finally conclude that the Young measure $\{ \Nu_{t,x} \}_{t,x \in (0,T) \times \Omega}$ represents
a DMV solution of the Navier--Stokes system \eqref{strong_ns}
in the sense of Definition~\ref{def_dmvs} which proves the first result of Theorem~\ref{thm_convergence}. The second result is an immediate consequence of Theorem~\ref{thm_ws_principle}.
\qed
%

\section{Proof of Theorem \ref{thm_error_estimates}: convergence rate}\label{sec:error}

The proof of error estimates via the relative energy functional requires the derivation of three estimates, namely 
\begin{itemize}[noitemsep,topsep=0pt,leftmargin=25pt]
\item[i)] consistency error: the identity (inequality) satisfied by a strong solution;
\item[ii)] discrete relative energy: the discrete counterpart of the continuous version of the relative energy inequality;
\item[iii)] approximate relative energy: approximation of the discrete relative energy with a particularly chosen discrete test functions and suitably transformed terms.
\end{itemize}
Application of the Gronwall inequality on a suitable combination of the approximate  relative energy inequality and the consistency error shall yield the desired convergence rate at the end of this section.

We start the proof by reporting the consistency error satisfied by a strong solution from \cite[Lemma 7.1]{GallouetMAC} (see also \cite[Lemma 7.1]{Gallouet_mixed}). 
\begin{lemma}[Consistency error] \label{lem_F1}
Let $(\vrh, \vuh) \in \Xspace \times \Yspace$ and $\vuh$ satisfy the estimates \eqref{est_u}, i.e., 
\[ \norm{\dih \vu_h}_{L^2L^2} \aleq 1, \quad 
\norm{\Gradedged \vuh}_{L^2 L^2}  \aleq 1,\quad  
\norm{\vuh}_{L^{2}L^6}  \aleq 1. \]
 Let $(r, \vU)$ be a solution of the Navier--Stokes system \eqref{strong_ns} that belongs to the class \eqref{ST_class}. Then for any $m=1,\ldots, N_t$ and $(r_h, \vUh) := (\Pim r,\Pie \vU)$ the following identity holds true:
\[ \sum_{i=1}^5J_i +\mathcal{R}_h^m =0,
\]
where $| \mathcal{R}_h^m | \aleq h + \TS$ and 
\begin{equation}\label{FJ}
\begin{aligned}
J_1&= \TS \sum_{n=1}^m  \mu \intO{\Gradedged \vU_h^n : \Gradedged(\vu_h^n - \vU_h^n) }
+ \TS \sum_{n=1}^m  (\mu+\lambda) \intO{\dih \vU_h^n : \dih(\vu_h^n - \vU_h^n) } , 
\\ J_2 &= \TS \sum_{n=1}^m \intO{ r_h^{n-1} D_t \avU _h^n  \left(  \avu_h^{n} - \avU_h^{n}\right) } ,
\\
J_3 &=  \TS \sum_{n=1}^m \sumK \sumfaceK \intS{ r_h^{\nup} \left( \avu_h^{\nup} -\avU_h^{\nup}  \right)\cdot (\vU_h^n - \avU_h^n )(\Udnup) } ,
\\
J_4 &=  \TS  \sum_{n=1}^m  \intO{  p(r_h^n) [\dix \vU]^n },
\quad J_5 = \TS \sum_{n=1}^m \intO{ p'(r_h^n) \avu_h^n \cdot [\Grad r]^n }.
\end{aligned}
\end{equation}
\end{lemma}

Next, recalling \cite[Lemma 8.1]{GallouetMAC} (see also \cite[Lemma 8.1]{Gallouet_mixed}) we have the following estimates. 
\begin{lemma} \label{lem_F2}
Let $(\vrh,\vuh)\in\Xspace\times\Yspace$ satisfy the uniform bounds stated in Corollary \ref{ests}. Let $(r,\vU)$ be a solution to the Navier--Stokes system \eqref{strong_ns} that belongs to the class \eqref{ST_class}. Then it holds that
\[ \sum_{i=1}^6 Q_i + \sum_{i=1}^5 J_i =\mathcal{P}_h, \]
where $\abs{\mathcal{P}_h } \leq c  \TS \sum_{n=1}^m 
\mathfrak{E}(\vr_h^n,\vu_h^n | r_h^n, \vUh^n) +c \delta \TS \sum_{n=1}^m \norm{\Gradedged (\vuh^n -\vUh^n)}_{L^2} $ with $\delta$ sufficiently small with respect to $\mu$, and 
\begin{equation}\label{FQ}
\begin{aligned}
Q_1 &= - \TS \sum_{n=1}^m \mu \intO{\Gradedged  \vUh^n : \Gradedged(\vu_h^n -  \vUh^n) }- \TS \sum_{n=1}^m (\mu+\lambda) \intO{\dih  \vUh^n : \dih(\vu_h^n -  \vUh^n) } ,
\\
Q_2 &= \TS \sum_{n=1}^m \intO{ \vr_h^{n-1} D_t  \avU_h^n \left(   \avU_h^n - \avu_h^{n} \right) } ,
\\
Q_3 &=  \TS \sum_{n=1}^m \sumK \sumfaceK \intS{ \vr_h^{\nup} \left( \avU_h^{\nup}- \avu_h^{\nup}  \right)\cdot ( \vUh^n - \avU_h^n)(\Udnup) } ,
\\
Q_4 &= - \TS  \sum_{n=1}^m  \intO { p(\vr_h^n) [\dix \vU]^n } ,
\quad
Q_5 = \TS \sum_{n=1}^m \intO{ (r_h^n - \vr_h^n)  \frac{p'(r_h^n)}{r_h^n} [\pd_t r]^n } ,
\\
Q_6 &= - \TS \sum_{n=1}^m \intO{ \vr_h^n  \frac{p'(r_h^n)}{r_h^n} \avu_h^n \cdot [\Grad r]^n  } .
\end{aligned}
\end{equation}
\end{lemma}

\subsection{Exact relative energy inequality}
In this section, we derive the relative energy on the discrete level. 
\begin{lemma}[Discrete relative energy]\label{lem_F3}
Let $(\vrh,\vuh) \in \Xspace \times \Yspace $ be a solution to the MAC scheme \eqref{scheme}. Then for any $(r_h,\vUh)\in \Xspace\times\Yspace,$ $r_h >0,$ and for $m=1,\ldots N_t$ it holds 
\begin{equation}\label{re1}
\begin{aligned}
\intO{  \frac12 \left(\vr_h^m |\avu_h^m - \avU_h^m|^2  - \vr_h^0 |\avu_h^0 - \avU_h^0|^2 \right)}
&+ \intO{ \left( \mathbb{E}(\vr_h^m|r_h^m)  -  \mathbb{E}(\vr_h^0|r_h^0) \right) }
\\
   +\TS  \sum_{n=1}^m \mu \intO{ |\Gradedged (\vu_h^n - \vU_h^n)|^2 }
& +\TS  \sum_{n=1}^m (\mu+\lambda) \intO{ |\dih (\vu_h^n - \vU_h^n)|^2 }
 +  \sum_{i=1}^8 T_i \leq 0,
\end{aligned}
\end{equation}
where $T_i= \sum\limits_{n=1}^m \TS \; T_i^n$ and $T_i^n$ read
\[
\begin{aligned}
T_1^n &=   \mu \intO{\Gradedged \vU_h^n : \Gradedged(\vu_h^n - \vU_h^n) }
 + (\mu+\lambda) \intO{\dih \vU_h^n : \dih(\vu_h^n - \vU_h^n) }
 , 
\\ 
T_2^n &= -  \intO{ \vr_h^{n-1} D_t \avU _h^n  \left( \frac{\avU_h^{n-1} + \avU_h^{n} }{2} - \avu_h^{n-1} \right) } ,
\\
T_3^n &=  \sumK \sumfaceK \intS{\vr_h^{\nup} \left( \avs{\avU_h^n} - \avu_h^{\nup}  \right)\cdot \avU_h^n (\udn)} ,
\\
T_4^n &=   \sumK \sumfaceK \intS{ p(\vr_h^n)  (\Udn)} ,
\\
T_5^n &=- \intO{ \frac{r_h^n - \vr_h^n}{\TS} \big(H'(r_h^n) -H'(r_h^{n-1}) \big) } ,
\\
T_6^n &= -  \sumK \sumfaceK \intS{ \vr_h^{\nup} H'(r_h^{n-1})  (\udn) } ,
\\
T_7^n &=  h^\alpha  \intO{ \Lapm \vr_h^n H'(r_h^{n-1})  } ,
\\
T_8^n &= h^\alpha  \intO{\Gradedge \vr_h^n \cdot \Gradedge \avU_h^n \cdot \avs{ \avU_h^n-\avu_h^n} }.
\end{aligned}
\]
\end{lemma}
\begin{proof}
We multiply the discrete density equation \eqref{scheme_D} by $\frac 12\left(|\avU_h^n|^2-|\avu_h^n|^2 \right)$ and $\left(H'(\vr_h^n)-H'(r_h^{n-1})\right)$. Then we multiply the discrete momentum equation \eqref{scheme_M} by $\left(\uih^n-\Uih^n\right)$, and sum over $i=1,\ldots,d.$  We integrate the resulting equations over $\Omega$ to get 
\begin{align*}
0&=\intO{D_t \vrh^n  \frac{|\avU_h^n|^2 - |\avu_h^n|^2}{2}} + \intO{\di_{\Up} [\vr^n_h, \vu_h^n]  \frac{|\avU_h^n|^2 - |\avu_h^n|^2}{2}} - h^\alpha \intO{\laph \vrh^n  \frac{|\avU_h^n|^2 - |\avu_h^n|^2}{2}}
\\&= :\sum_{k=1}^3 I_k, \\
0&=\intO{D_t \vrh^n \left( H'(\vr_h^n)-H'(r_h^{n-1})\right)} + \intO{\di_{\Up} [\vr^n_h, \vu_h^n] \left( H'(\vr_h^n)-H'(r_h^{n-1})\right)} 
\\& \quad 
- h^\alpha \intO{\laph \vrh^n  \left( H'(\vr_h^n)-H'(r_h^{n-1})\right)} = :\sum_{k=4}^6 I_k, \\
 0&=\intO{D_t \avs{ \vrh^n \Ov{\vu_h^n} }\cdot\left( \vu_h^n-\vU_h^n\right)} + \intO{\avs{\di_\Up[\vrh^n \Ov{\vu_h^n}, \vu_h^n] }\cdot\left( \vu_h^n-\vU_h^n\right)} 
+ \intO{\Gradedge p (\vrh^n)\cdot \left( \vu_h^n-\vU_h^n\right)} 
\\ & \quad 
- \intO{ \big( \mu  \Lape \vu_h^n + (\mu+\lambda) \Gradedge (\dih \vu_h^n)\big) \cdot\left( \vu_h^n-\vU_h^n\right) }  
\\& \quad 
-  h^\alpha \sumi \intO{ \sumj \avs{ \pdmesh^{(j)} \left( \avs{\Ov{\uih^n}}^{(j)}  (\pdedge^{(j)} \vrh^n)  \right) }^{(i)}  \left( \uih^n-\Uih^n\right)}  =:\sum_{k=7}^{11} I_k.
\end{align*}
Now we sum up all $I_k$ terms and derive the desired inequality in 7 steps:
\begin{itemize}
\item  The sum of $I_1$ and $I_7$ yields $T_2^n$:
\begin{align*}
I_1+I_7 &= \frac{1}{2\TS} \intO{ \left[ \vr_h^n |\avu_h^n-\avU_h^n|^2 -
 \vr_h^{n-1} |\avu_h^{n-1}-\avU_h^{n-1}|^2 \right]} \\
 &+ \frac{1}{2\TS}\intO{\vr_h^{n-1}\bigg(|\avu_h^{n-1}-\avU_h^{n-1}|^2+|\avu_h^{n}|^2-|\avU_h^{n}|^2 - 2 \avu_h^{n-1}\avu_h^n +2 \avu_h^{n-1}\avU_h^n\bigg)}\\
 & = \frac{1}{2\TS} \intO{ \left[ \vr_h^n |\avu_h^n-\avU_h^n|^2 -
 \vr_h^{n-1} |\avu_h^{n-1}-\avU_h^{n-1}|^2 \right]}  +  T_2^n + D_1,
\end{align*}
where 
\begin{align*}
D_1=\frac{1}{2\TS}\intO{\vr_h^{n-1}|\avu_h^{n-1}-\avu_h^{n}|^2} \geq 0.
\end{align*}
\item   Term $I_4$ results in $T_5^n$:
\begin{align*}
I_4 &= \frac{1}{\TS}\intO{\left[H(\vr_h^n)-H(\vr_h^{n-1})+\frac12 H'(\vr_h^*)(\vr_h^n-\vr_h^{n-1})^2\right]} + \frac{1}{\TS}\intO{(\vr_h^{n-1}-\vr_h^{n})H'(r_h^{n-1})}\\
&= \frac{1}{\TS}\intO{\mathbb{E}(\vr_h^n| r_h^n)- \mathbb{E}(\vr_h^{n-1}| r_h^{n-1})}+  T_5^n + D_{2,1} + D_{2,2},
\end{align*}
where 
\begin{align*}
D_{2,1}&= \frac{1}{2\TS}\intO{ H'(\vr_h^*)(\vr_h^n-\vr_h^{n-1})^2} \geq 0,\\
D_{2,2}&=\frac{1}{\TS}\intO{\bigg[H(r_h^n)-H(r_h^{n-1})-H'(r_h^{n-1})(r_h^n-r_h^{n-1})\bigg]}=\frac{1}{2\TS}\intO{ H''(r_h^*)(r_h^n-r_h^{n-1})^2} \geq 0.
\end{align*}

\item Term $T_3^n$ comes from adding $I_2$ and $I_8$ together:
\begin{align*}
I_2 + I_8 & = \sumK \sumfaceK \intS{ \left[ \vr_h^{\nup} \avu_h^{\nup} \cdot (\avu_h^n - \avU_h^n) - \vr_h^{\nup} \frac{|\avu_h^n|^2 - |\avU_h^n|^2}{2} \right] \udn }
\\& = \shkl \intS{ \vr_K^n [\udn]^+ \left( \frac12 |\avu_K^n -\avu_L^n|^2  +  (\avU_K^n -\avU_L^n) \left(\frac{\avU_K^n +\avU_L^n}{2}- \avu_K^n \right) \right) }
\\& \quad+
 \shkl \intS{ \vr_L^n [\vuh^n\cdot \bfn_{\sigma,L}]^+ \left( \frac12 |\avu_L^n -\avu_K^n|^2  +  (\avU_L^n -\avU_K^n) \left(\frac{\avU_L^n +\avU_K^n}{2}- \avu_L^n \right) \right) }
\\&
= T_3^n +D_3,
\end{align*}
where 
\[ D_3 = \frac12 \shkl \intS{  \vrh^{\nup} |\avu_K^n - \avu_L^n|^2  |\udn| } \geq 0.
\]

\item  The sum of $I_5$ and $I_9$ yields both $T_4^n$ and $T_6^n$:
\begin{align*}
I_5 + I_9 & =
 \sumK \sumfaceK \intS{ \left[ \vr_h^{\nup} (H'(\vrh^n)- H'(r_h^{n-1})) \udn -  p(\vrh^n) (\vuh^n - \vUh^n) \cdot \bfn_{\sigma,K} \right] }
\\& =  T_6^n + T_4^n + D_4,
\end{align*}
where 
\begin{align*}
D_4 &=  \sumK \sumfaceK \intS{ \big( \vr_h^{\nup} H'(\vrh^n) - p(\vrh^n) \big)  \udn}
\\& =
\shkl \intS{ [\udn]^+ ( H(\vr_K^n) - H'(\vr_L^n)(\vr_K^n- \vr_L^n) -  H(\vr_L^n)  ) }
\\& \quad 
+ \shkl \intS{[\vuh^n \cdot \bfn_{\sigma,L}]^+ ( H(\vr_L^n) - H'(\vr_K^n)(\vr_L^n- \vr_K^n) -  H(\vr_K^n)  )  }
 \geq 0.
\end{align*}

\item The sum of $I_3$ and $I_{11}$ gives exactly $T_8^n$:
\begin{align*}
I_3 + I_{11} &= - h^\alpha \intO{ \Lapm \vrh^n \frac{|\avU_h^n|^2 - |\avu_h^n|^2}{2}
- \ha \sumi \avs{ \sumj \pdmeshj\big( \avs{\Ov{u}_{i,h}^n}^{(j)} \pdedgej \vrh^n \big)}^{(i)} (\uih^n -\Uih^n ) }
\\& =
\frac{\ha}{2} \intO{ \Gradedge \vrh^n \cdot \Gradedge \left( |\avU_h^n|^2 - |\avu_h^n|^2\right)}
+ \ha \sumi \sumj \intO{   \avs{\Ov{\uih^n}}^{(j)}  \pdedgej \vrh^n \pdedgej(\Ov{\uih^n}-\Ov{\Uih^n}) }
\\& =
\ha \sumi \sumj \intO{ \pdedgej \vrh^n \pdedgej \Ov{\Uih^n} \avs{\Ov{\Uih^n}}^{(j)} }
- \ha \sumi \sumj \intO{  \avs{\Ov{\uih^n}}^{(j)} \pdedgej \vrh^n \pdedgej \Ov{\Uih^n} }
= T_8^n.
 \end{align*}
 
\item Term $I_6$ results in $T_7^n$:
\begin{align*}
I_6 =  T_7^n + D_5, 
\end{align*}
where
\[ D_5=\ha  \sumK \sumfaceK \intS{ \Gradedge \vrh^n \cdot \Gradedge H'(\vrh^n)}
=\ha  \sumK \sumfaceK \intS{ |\Gradedge \vrh^n|^2 H''(\vr^n_{h,\dagger})} \geq 0.
\]

\item Finally, by rewriting $I_{10}$ in a convenient way we get $T_1^n$:
\begin{align*}
I_{10} 
=  \mu \intO{ |\Gradedged (\vu_h^n - \vU_h^n)|^2 } 
 +  (\mu+\lambda) \intO{ |\dih (\vu_h^n - \vU_h^n)|^2 }
+ T_1^n.
\end{align*}
\end{itemize}

 Collecting all the above calculations and summing them up for all times steps $i=1,\ldots,m$ finishes the proof. 

\end{proof}

\subsection{Approximate relative energy inequality}
In this subsection, we further analyse the inequality derived in Lemma \ref{lem_F3} (for the numerical solution). The aim is to derive the $Q_i$ terms stated in Lemma~\ref{lem_F2} (for the strong solution) from the $T_i$ terms, such that we can use the result of Lemmas \ref{lem_F1} and \ref{lem_F2} to estimate the relative energy between the numerical and strong solutions. To this end, the test function pair $(r_h,\vUh)$ in Lemma~\ref{lem_F3} must be chosen properly, see the result below.  
\begin{lemma}[Approximate relative energy]\label{lem_F4}
Let $(\vrh,\vuh) \in \Xspace\times\Yspace$ be a solution to scheme \eqref{scheme}, and let 
$(r_h,\vUh) :=(\Pim r, \Pie \vU)$ for $(r,\vU)$ be a strong solution to the system \eqref{strong_ns} that belongs to the class \eqref{ST_class}. 
Then there exists a positive constant
 \[c=c(M_0, E_0, \underline{r}, \overline{r}, |p'|_{C([\underline{r},\overline{r}])}, \|(\Grad r, \pd_t r, \pd_t \Grad r, \pd_t^2 r, \vU, \Grad \vU, \Grad^2 \vU, \pd_t\vU, \pd_t \Grad \vU)\|_{L^\infty(\Omega)})\]
such that for all $m=1,\ldots,N_t$ it holds  
\begin{equation*}
\begin{aligned}
\mathfrak{E}(\vr_h^m , \vu_h^m| r_h^m, \vUh^m)    -
\mathfrak{E}(\vr_h^0 , \vu_h^0| r_h(0), \vUh(0)) 
+\TS \sum_{n=1}^m \mu \intO{  |\Gradedged (\vuh^n -  \vUh^n)|^2 }
\\ +\TS \sum_{n=1}^m(\mu+\lambda) \intO{ |\dih (\vu_h^n - \vU_h^n)|^2 }
\leq \sum_{i=1}^6 Q_i + R_h^m + G^m,
\end{aligned}
\end{equation*}
where $Q_i, i=1,\ldots, 6$ are given in \eqref{FQ} 
and
\[
\abs{G^m} \leq c \TS \sum_{n=1}^m \mathfrak{E}(\vrh^n , \vuh^n| r_h^n, \vUh^n), \quad 
\abs{R_h^m} \leq c(\sqrt{\TS} + h^A), \quad 
A=\begin{cases} \frac{2\gamma - d}{\gamma} & \mbox{ if } \gamma \in [\frac32,2),\\
\frac12 & \mbox{ if } \gamma \geq 2.
\end{cases}
\]
\end{lemma}
\begin{proof}
We start the proof from the inequality \eqref{re1} derived in the previous Lemma~\ref{lem_F3}. We only need to deal with the terms $T_i,$ $i=1,\ldots,8,$  as the other terms will remain the same. 

\begin{itemize}
\item
We keep the term $T_1$  unchanged and set $Q_1 = - T_1$.  

\item  The second term $T_2$ can be rewritten as 
\begin{align*}
- T_2  &=
\TS \ \sum_{n=1}^m \intO{ \vr_h^{n-1} D_t \avU _h^n  \left( \frac{\avU_h^{n-1} + \avU_h^{n} }{2} - \avu_h^{n-1} \pm \frac12 \avU_h^n \pm \avu_h^n \right) } = 
\\ & = 
Q_2
+\underbrace{ \frac12 \TS \ \sum_{n=1}^m \intO{ \vr_h^{n-1} D_t \avU _h^n  (\avU_h^{n-1} - \avU_h^{n}) }
+  \TS \ \sum_{n=1}^m \intO{ \vr_h^{n-1} D_t \avU _h^n  (\avu_h^{n} - \avu_h^{n-1}) }  }
_{ =: R_1 },
\end{align*}
where by the interpolation estimate \eqref{NC3} and the uniform bounds \eqref{ests1} we have 
\begin{align*}
\abs{ R_{1} } &= \abs{ \frac12 \TS \ \sum_{n=1}^m \intO{ \vr_h^{n-1} D_t \avU _h^n  (\avU_h^{n-1} - \avU_h^{n}) }
+  \TS \ \sum_{n=1}^m \intO{ \vr_h^{n-1} D_t \avU _h^n  (\avu_h^{n} - \avu_h^{n-1}) }  }
\\& \aleq
\TS \norm{\vrh}_{L^{\infty} L^1} \norm{\pd_t U}_{L^{\infty}W^{1,\infty}}^2
+ \norm{\vrh}_{L^{1} L^1}^{1/2}  \norm{\pd_t U}_{L^{\infty}W^{1,\infty}}  
\left( \int_0^T \intO{ \vr_h^{n-1} \frac{|\avu_h^{n} - \avu_h^{n-1}|^2}{(\TS)^2} (\TS)^2 } \dt \right)^{1/2}
\\& \aleq c(E_0, \norm{\vU}_{C^2} ) \TS.
\end{align*}

\item From the third term $T_3$ we get
\begin{align*}
- T_3 &= - \TS \sum_{n=1}^m \sumK \sumfaceK \intS{\vr_h^{\nup} \left( \avs{\avU_h^n} - \avu_h^{\nup}  \right)\cdot \avU_h^n (\udn)} 
\\&=  
- \TS \sum_{n=1}^m \sumK \sumfaceK \intS{\vr_h^{\nup} \left( \avU_h^{\nup} - \avu_h^{\nup}  \right)\cdot \avU_h^n (\udn)} +R_{2,1},
\end{align*}

where 
\[
\begin{aligned}
R_{2,1} &= 
 \TS \sum_{n=1}^m \sumK \sumfaceK \intS{\vr_h^{\nup} \left( \avU_h^{\nup} - \avs{\avU_h^n}  \right)\cdot \avU_h^n (\udn)} 
\\& = 
\frac{\TS}2 \sum_{n=1}^m \shkl \intS{ \Big[ \vr_K^n |\avU_K^n - \avU_L^n |^2[\udn]^+  +
\vr_L^n |\avU_K^n - \avU_L^n |^2[\vuh^n \cdot \bfn_{\sigma,L}]^+  \Big] }.
\end{aligned} 
\]
Seeing the equality 
\[ \sumK \sumfaceK \intS{f_h^{up} \cdot \vUh^n (\udn)}  =0 \mbox{ for } f_h^{up} = \vr_h^{\nup} \left( \avU_h^{\nup} - \avu_h^{\nup}  \right)
\]
we have 
\begin{align*} - &T_3 = \TS \sum_{n=1}^m \sumK \sumfaceK \intS{\vr_h^{\nup} \left( \avU_h^{\nup} - \avu_h^{\nup}  \right)\cdot (\vUh^n - \avU_h^n) (\udn)}  +R_{2,1} 
\\ & = 
\TS \sum_{n=1}^m \sumK \sumfaceK \intS{ \vr_h^{\nup} \left( \avU_h^{\nup}- \avu_h^{\nup}  \right)\cdot ( \vUh^n - \avU_h^n)(\Udnup) } +R_{2,1} +R_{2,2} 
\\& = Q_3 + R_{2,1} +R_{2,2},
\end{align*}
where 
\begin{align*}
R_{2,2} &=  \TS \sum_{n=1}^m \sumK \sumfaceK \intS{\vr_h^{\nup} \left( \avU_h^{\nup} - \avu_h^{\nup}    \right)\cdot (\avU_h^n -\vUh^n) (\avU^{\nup}_h-\vuh^n )\cdot \bfn_{\sigma,K}}
\\&  =
  \TS \sum_{n=1}^m \sumK \sumfaceK \intS{\vr_h^{\nup} \left( \avU_h^{\nup} - \avu_h^{\nup}    \right)\cdot (\avU_h^n -\vUh^n) (\avU^{\nup}_h- \avu_h^{\nup})\cdot \bfn_{\sigma,K}}
\\ & \quad  +
   \TS \sum_{n=1}^m \sumK \sumfaceK \intS{\vr_h^{\nup} \left( \avU_h^{\nup} - \avu_h^{\nup}    \right)\cdot (\avU_h^n -\vUh^n) ( \avu_h^{\nup} -\vuh^n  )\cdot \bfn_{\sigma,K}} 
\\& = R_{2,2,1}+R_{2,2,2}.
\end{align*}
Given  $\gamma \geq \frac{d}{2} >\frac65$, we get
\begin{align*}
\abs{ R_{2,1} }\aleq h \norm{\vrh}_{L^2 L^{6/5}}  \norm{\vuh}_{L^2 L^6}  (\norm{\Grad \vU}_{L^{\infty} L^{\infty}} )^2 \aleq c(E_0, \norm{\vU}_{C^1}) h,
\end{align*}
by using H\"older's inequality, the uniform bounds \eqref{ests1} and the trace inequality \eqref{ineq_trace}. 
Further, by a similar argument it holds 
\begin{align*}
\abs{ R_{2,2,1} } \aleq c(\norm{\vU}_{C^1}) \TS\sum_{n=1}^m \mathfrak{E}(\vrh^n,\vuh^n|r^n,\vU^n) ,
\end{align*}
and 
\begin{align*}
\abs{ R_{2,2,2} } &\aleq \norm{\vU}_{C^1} \norm{\sqrt{\vrh}}_{L^{\infty}L^{2\gamma}} \TS\sum_{n=1}^m \mathfrak{E}(\vrh^n,\vuh^n|r^n,\vU^n)^{1/2}\norm{ \avu_h^{\nup} -\vuh^n}_{L^q}\\
&\aleq \norm{\vU}_{C^1}  \norm{\vrh}_{L^{\infty}L^{\gamma}}^{1/2} \left(
h^{2d\left(\frac{1}{q}-\frac12\right)} h^2  \norm{ \Gradedged\vu_h}^2_{L^2L^2} +\TS\sum_{n=1}^m \mathfrak{E}(\vrh^n,\vuh^n|r^n,\vU^n)\right)\\
&\aleq  c(E_0, \norm{\vU}_{C^1}) h^{\frac{2\gamma-d}{\gamma}} + c(E_0, \norm{\vU}_{C^1}) \TS\sum_{n=1}^m \mathfrak{E}(\vrh^n,\vuh^n|r^n,\vU^n),
\end{align*}
where we have also used Young's inequality. Here $q=\frac{2\gamma}{\gamma-1} \in (2,6)$ provided $\gamma >\frac{d}{2}.$

\item
The fourth term $T_4$ directly yields
\begin{align*}
- T_4 &= - \TS \sum_{n=1}^m \sumK \sumfaceK \intS{ p(\vr_h^n) \dix \vUh}
 = - \TS \sum_{n=1}^m \sumK \sumfaceK \intS{ p(\vr_h^n) [\dix \vU ]^n } \\
 &=Q_4.
\end{align*}

\item We proceed with the fifth term $T_5$ and obtain
\begin{align*}
- T_5 &=  \TS \sum_{n=1}^m \intO{ \frac{r_h^n - \vr_h^n}{\TS} \big(H'(r_h^n) -H'(r_h^{n-1}) \big) } 
\\ & 
=\TS \sum_{n=1}^m \intO{ \frac{r_h^n - \vr_h^n}{\TS} \left( H''(r_h^n) ( r_h^n- r_h^{n-1}) - \frac{H'''(r^{n,\star}_h)}{2}( r_h^n- r_h^{n-1})^2  \right)}
\\& = Q_5 + R_{3,1} +R_{3,2},
\end{align*}
where 
\[
R_{3,1}= - \TS \sum_{n=1}^m \intO{ \frac{r_h^n - \vr_h^n}{\TS}  \frac{H'''(r^{n,\star}_h)}{2}( r_h^n- r_h^{n-1})^2 },
\]
\[
R_{3,2} = \TS \sum_{n=1}^m \intO{ (r_h^n - \vr_h^n) \frac{p'(r_h^n)}{r_h^n} ( D_t r_h^n - [\pd_t r]^n)}.
\]
The two residual terms can be estimated as follows
\begin{align*}
|R_{3,1}| &\aleq  \TS \norm{r_h - \vrh}_{L^{1}L^1}  |p'|_{C^1([\underline{r}, \Ov{r}])} \norm{\pd_t r}_{L^{\infty}L^{\infty}} 
\aleq c(M_0, \norm{p}_{C^2([\underline{r}, \Ov{r}])}, \norm{r}_{C^1} ) \TS,\\
\abs{ R_{3,2} }  &\aleq  \TS \norm{r_h - \vrh}_{L^{\infty}L^1}  |p'|_{C^1([\underline{r}, \Ov{r}])} \big( \norm{\pd_{t}^2 r}_{L^{\infty}L^{\infty}} +  \norm{\pd_{t} \Grad r}_{L^{\infty}L^{\infty}} \big).
\end{align*}

\item Term $T_6$ yields $Q_6$ after a suitable manipulation and estimating three residual terms. Indeed, 
\begin{align*}
- T_6 &= \TS \sum_{n=1}^m \sumK \sumfaceK \intS{ \vrh^{\nup} H'(r_h^{n-1})  (\udn) } ,
\\& = \TS \sum_{n=1}^m \sumK \sumfaceK \intS{ \vrh^n \big( H'(r_h^{n-1}) -H'(\Pie r ^{n-1}) \big)  (\udn) } 
\\ & \quad 
+\underbrace{ \TS \sum_{n=1}^m \sumK \sumfaceK \intS{ (\vrh^{\nup} -\vrh^n )  \big( H'(r_h^{n-1}) -H'(\Pie r ^{n-1}) \big)  (\udn) } }_{R_{4,1}}
\\& = 
\TS \sum_{n=1}^m \sumK \sumfaceK \intS{ \vrh^n  H''(r_h^{n-1}) (r_h^{n-1}- \Pie r ^{n-1})   (\udn) } +R_{4,1}
\\& \qquad 
- \underbrace{ \frac{\TS}{2} \sum_{n=1}^m \sumK \sumfaceK \intS{ \vrh^n  H'''(r^{n}_{h,\dagger})(r_h^{n-1}- \Pie r ^{n-1})^2  (\udn) } }_{R_{4,2}}
\\ & =
- \TS \sum_{n=1}^m  \intO{ \vrh^n  H''(r_h^{n-1}) \avu_h^n \cdot [\Grad r]^{n-1} } +R_{4,1} +R_{4,2}  
\\& \qquad 
+ \underbrace{ \TS \sum_{n=1}^m \sumK \sumfaceK \intS{ \vrh^n  H''(r_h^{n-1}) (r_h^{n-1}- \Pie r ^{n-1})   (\avu_h^n - \vuh^n) \cdot \bfn_{\sigma,K} }   }_{=:R_{4,3}} 
\\& = Q_6 + R_{4,1} +R_{4,2}  + R_{4,3}  ,
\end{align*}
where we have used the following equality in the last second line
\[  \sumfaceK \intS{ (r_K -  \Pie r ) \avu_h \cdot \bfn_{\sigma,K} }= - \int_K{ \avu_h \cdot [\Grad r] } \dx.
\]

Now we estimate the residual terms $R_{4,i}, i=1,2,3$.  It holds
\begin{equation}\label{R41}
\begin{aligned}
\abs{ R_{4,1} } & = \abs{ \TS \sum_{n=1}^m \sumK \sumfaceK \intS{ (\vrh^{\nup} -\vrh^n )  \big( H'(r_h^{n-1}) -H'(\Pie r ^{n-1}) \big)  (\udn) }  }
\\& \aleq 
\TS  \sum_{n=1}^m \sumK  \sumfaceK \intS{ |(\vrh^{\nup} -\vrh^n)\udn|\;  H''(r_{h,\dagger}^{n-1})  |r_h^{n-1} - \Pie r ^{n-1}| } 
\\ & \aleq
 h |p|_{C^1([\underline{r}, \Ov{r}])} \norm{\Grad r}_{L^{\infty}L^{\infty} } 
\int_0^T \sumK \sumfaceK \int_{D_\sigma}{   \abs{ \Gradedge \vrh   (\vuh\cdot\bfn)^-} \dx \dt} 
\\ & \aleq
c\left( E_0, |p|_{C^1([\underline{r}, \Ov{r}])}, \norm{\Grad r}_{L^{\infty}L^{\infty} }   \right) h^{1/2},
\end{aligned}
\end{equation}
where we have used Lemma~\ref{lem_S1}. Further, given $\gamma \geq \frac{d}{2}>\frac65,$ we get
\begin{align*}
\abs{ R_{4,2}  }
& =\frac{\TS}{2} \sum_{n=1}^m \sumK \sumfaceK \intS{ \vrh^n  H'''(r^{n}_{h,\dagger})(r_h^{n-1}- \Pie r ^{n-1})^2  (\udn) }\\
& \aleq 
h |p'|_{C^1([\underline{r}, \Ov{r}])} \norm{\Grad r}_{L^{\infty}L^{\infty} }^2  \norm{\vrh}_{L^2L^{6/5}}  \norm{\vuh}_{L^2L^6} 
\\& \leq 
c\left( E_0,  \norm{p'}_{C^1([\underline{r}, \Ov{r}])} , \norm{\Grad r}_{L^{\infty}L^{\infty} }^2  \right) h .
\end{align*}
Finally, by using H\"older's inequality, the trace inequality \eqref{ineq_trace}, the velocity bounds \eqref{est_u}, and Lemma~\ref{lem_est_l2l2} we get
\begin{align*}
\abs{ R_{4,3}  } = 
&\abs{ \TS \sum_{n=1}^m \sumK \sumfaceK \intS{ \vrh^n  H''(r_h^{n-1}) (r_h^{n-1}- \Pie r ^{n-1})   (\avu_h^n - \vuh^n) \cdot \bfn_{\sigma,K} }   }\\
& \aleq 
h |p'|_{C^0([\underline{r}, \Ov{r}])} \norm{\Grad r}_{L^{\infty}L^{\infty} } 
\norm{\vrh }_{L^2L^2}
\norm{\Gradedged \vuh}_{L^2L^2}
\\ & \aleq 
c\left(E_0, |p'|_{C^0([\underline{r}, \Ov{r}])} \norm{\Grad r}_{L^{\infty}L^{\infty} } \right) h^\beta,
\end{align*}
where  
\begin{equation}\label{hbeta}
 \beta = \begin{cases}
 1+ \frac{\gamma -2 }{2 \gamma }d >\frac{2\gamma-d}{\gamma} &\mbox{ if } \gamma \in [\frac{d}{2},2), \\
1 >\frac12 & \mbox{ if } \gamma \geq 2. 
\end{cases}
\end{equation}
\item
For the seventh term $T_7$ we get
\begin{align*}
| T_7| & =  \left| \TS h^\alpha  \sum_{n=1}^m \intO{ \Lapm \vr_h^n H'(r_h^{n-1})  } \right|
 =
\left| \TS h^\alpha  \sum_{n=1}^m  \intO{ - \Gradedge \vrh^n \cdot \Gradedge H'(r_h^{n-1}) } \right| 
\\&=  
\left| \TS h^\alpha \sum_{n=1}^m \intO{  \vr_h^n \Lapm H'(r_h^{n-1})  } \right|
\\& \aleq h^\alpha  |p'|_{C^1([\underline{r}, \Ov{r}])} \norm{\Grad r}_{L^{\infty}L^{\infty} } \norm{ \vrh }_{L^1 L^1}
 \aleq h^\alpha  |p'|_{C^1([\underline{r}, \Ov{r}])} \norm{\Grad r}_{L^{\infty}L^{\infty} } \norm{ \vrh }_{L^\infty L^\gamma}
\\& 
\leq  c(E_0, |p'|_{C^1([\underline{r}, \Ov{r}])}, \norm{\Grad r}_{L^{\infty}L^{\infty} }) \; \ha.
\end{align*}

\item The last term $T_8$ yields the bound
\begin{align*}
|T_8| &= \TS  \ha \left| \sum_{n=1}^m
 \sumi \sumj \intO{ \pdedgej \vrh^n \pdedgej \Ov{\Uih^n} \avs{\Ov{\Uih^n -\uih^n }}^{(j)} } \right|
\\& =   \TS  h^\alpha \left|  \sum_{n=1}^m  \sumi \sumj \intO{ \vr_h^n \pdmeshj \left( \pdedgej \Ov{\Uih^n}  \avs{ \Ov{\Uih^n} -\Ov{\uih^n} }^{(j)} \right)}  \right|
\\& =  
 \TS  h^\alpha  \left| \sum_{n=1}^m \intO{ \vr_h^n \left( \Lapm \avU_h^n \cdot \Ov{ \avs{ \avU_h^n-\avu_h^n}} \right)}  \right|
\\& \qquad +
\TS  h^\alpha  \left| \sum_{n=1}^m  \sumi \sumj 
 \intO{ \vr_h^n \Ov{\pdedgej \Ov{\Uih^n}} \cdot \pdmeshj \avs{ \Ov{\Uih^n} -\Ov{\uih^n} }^{(j)} }  \right|
\\& \aleq 
\ha \norm{ \vrh }_{L^{2} L^{6/5}} \norm{\vU}_{C^2} 
 \left( \norm{\vUh}_{C^0}  + \norm{\vuh}_{L^2 L^6} \right)
+ \ha 
\norm{ \vrh }_{L^1 L^1}   \norm{\vU}_{C^1}^2  +\ha \norm{\vU}_{C^1}
 \norm{\vrh }_{L^2 L^2}  \norm{ \Gradedged \vuh}_{L^2 L^2} 
\\& 
\aleq  c \left( \norm{\vU}_{C^2}, \norm{\vrh}_{L^\infty L^{\gamma}} ,\norm{\vuh}_{L^2 L^6}\right) \ha
+
c \left( \norm{\vU}_{C^1}, \norm{\vuh}_{L^\infty L^\gamma}, \norm{\Gradedged \vuh}_{L^2 L^2} \right) h^{\zeta},
\end{align*}
where we have used the inequality 
\[
\norm{ \pdmeshj \avs{ \Ov{\uih^n} }^{(j)} } _{L^2}  \aleq \norm{ \eth_j \uih} _{L^2},
\]
and $\zeta=\alpha -1 +\beta>A$ with $\beta$ being given in \eqref{hbeta} 
due to the same trick as in the estimate of the term $R_{4,3}$.
\end{itemize}
\end{proof}

\subsection{End of error estimates }
Collecting the estimates in Lemmas \ref{lem_F1}, \ref{lem_F2}, and \ref{lem_F4} immediately yields the following inequality
\begin{multline*}
\mathfrak{E}(\vr_h^m,\vu_h^m | r^m, \vU^m) + \TS \sum_{n=1}^m \frac{\mu}{2}\intO{|\Gradedged(\vu_h^n-\vU_h^n)|^2 }
 +\TS \sum_{n=1}^m(\mu+\lambda) \intO{ |\dih (\vu_h^n - \vU_h^n)|^2 }
 \\
\leq c \left[ h^A + \sqrt{\TS} + \mathfrak{E}(\vr_h^0, \vu_h^0 | r(0), \vU(0))\right] + c \TS \sum_{n=1}^m \mathfrak{E}(\vr_h^n, \vu_h^n | r^n, \vU^n),
\end{multline*}
for all $m=1,\ldots, N$. Here, the convergence rate $A$ is defined in \eqref{A}, and the positive constant 
\[c=c(M_0, E_0, \underline{r}, \overline{r}, |p'|_{C([\underline{r},\overline{r}])}, \|(\Grad r, \pd_t r, \pd_t \Grad r, \pd_t^2 r, \vU, \Grad \vU, \Grad^2 \vU, \pd_t\vU, \pd_t \Grad \vU)\|_{L^\infty(\Omega)})\]
 depends tacitly also on $T,$  $\gamma,$ \textnormal{diam}$(\Omega)$, $|\Omega|$.

Finally, applying Gronwall's inequality to the above estimates, we finish the proof of Theorem \ref{thm_error_estimates}.

\section{Numerical experiments}
Concerning the performance of scheme \eqref{scheme}, we refer to \cite[Section 5]{HS_MAC}, where both  the homogeneous Dirichlet and periodic boundary conditions were implemented. Here we aim to validate the theoretical results stated in Theorem~\ref{thm_error_estimates}, that is the convergence rate  derived in terms of the relative energy. Hence, we measure the following  errors
\begin{equation}\label{err_eoc}
\begin{aligned}
e_{\mathfrak{E}} = \sup \limits_{ 1\leq n \leq N_t} \mathfrak{E}(\vrh^n,\avu_h^n\vert r(t^n,\cdot),\vU(t^n,\cdot)),\quad 
e_{\nabla \vu}= \norm{ \Gradedged (\vuh-\vU) }_{L^2(0,T; \Omega)},\\
e_{\vr}=\norm{ \vrh-r }_{L^1(0,T; \Omega)}, \quad 
e_{\vu}=\norm{ \vuh-\vU }_{L^2(0,T; \Omega)},\quad 
e_{p}=\norm{ p(\vrh)-p(r)}_{L^{\infty}(0,T; \Omega)}
\end{aligned}
\end{equation}
between the numerical solution $(\vrh,\vuh)$ and the reference solution $(r,\vU).$ For this purpose we  perform two experiments in the domain $\Omega=[0,1]^2$. In the first experiment the reference solution is explicitly given by considering suitable external force in the momentum equation. In the second experiment the reference solution is set as the numerical solution computed on a very fine mesh.  In both tests, we set $\mu=1$, $\alpha=1.6$ satisfying \eqref{Choice_alpha}. 

\subsection{Experiment 1}
We first consider the following analytical solution  
\begin{equation}\label{ex_sol}
r(x,y,t) = 1 ,\quad 
\vU(x,y,t) =\left(
\begin{array}{l}
\sin(2\pi x) \cos(2\pi y)e^{-k t} \\
- \cos(2\pi x) \sin (2\pi y)e^{-k t}
\end{array} \right) ,\; k=0.01, 
\end{equation}
that is driven by the corresponding external force in the momentum equation. 
We show in Table~\ref{tab_exp1} the relative errors in the norms presented in \eqref{err_eoc} for different values of $\gamma$. Clearly, we observe the second order convergence rate for the relative energy and the first order convergence rate for the density, velocity and the gradient of velocity.  

\begin{table*}[h!]\centering
\caption{Experiment 1: error norms at $T=0.1$ for different $\gamma$}\label{tab_exp1}
\begin{tabular}{c|c|c|c|c|c|c|c|c|c|c}\hline 
  $h$   &  $e_{\mathfrak{E}} $ &  EOC & $ e_{\nabla \vu}$ &  EOC  &  $e_{\vr}$ &  EOC  &  $e_{\uu}$ &  EOC &  $e_{p }$ &  EOC  \\ \hline 
\multicolumn{11}{c}{ $\gamma=1.4$} \\ \hline 
  1/32 & 1.34e-02	 & -- & 5.03e-01	 & -- & 3.90e-03	 & --  & 4.31e-02	 & -- & 9.73e-02	 & -- \\  
1/64 &3.44e-03	 &1.96	 &2.53e-01	 &0.99	 &1.93e-03	 &1.02  &2.16e-02	 &0.99 &5.03e-02	 &0.95  \\  
1/128 &8.71e-04	 &1.98	 &1.27e-01	 &1.00	 &9.57e-04	 &1.01  &1.08e-02	 &1.00 &2.55e-02	 &0.98  \\  
1/256 &2.19e-04	 &1.99	 &6.35e-02	 &1.00	 &4.76e-04	 &1.01  &5.42e-03	 &1.00 &1.28e-02	 &0.99  \\  
\hline 
\multicolumn{11}{c}{ $\gamma=1.67$} \\ \hline 
  1/32 & 1.39e-02	 & -- & 5.02e-01	 & -- & 3.86e-03	 & --  & 4.28e-02	 & -- & 1.15e-01	 & -- \\  
1/64 &3.58e-03	 &1.96	 &2.53e-01	 &0.99	 &1.91e-03	 &1.01  &2.15e-02	 &0.99 &5.93e-02	 &0.95  \\  
1/128 &9.07e-04	 &1.98	 &1.27e-01	 &1.00	 &9.50e-04	 &1.01  &1.08e-02	 &1.00 &3.00e-02	 &0.98  \\  
1/256 &2.28e-04	 &1.99	 &6.34e-02	 &1.00	 &4.72e-04	 &1.01  &5.40e-03	 &1.00 &1.51e-02	 &0.99  \\  
\hline 
\multicolumn{11}{c}{ $\gamma=2$ } \\ \hline 
  1/32 & 1.45e-02	 & -- & 5.00e-01	 & -- & 3.82e-03	 & --  & 4.26e-02	 & -- & 1.36e-01	 & -- \\  
1/64 &3.75e-03	 &1.95	 &2.52e-01	 &0.99	 &1.90e-03	 &1.01  &2.14e-02	 &0.99 &7.02e-02	 &0.95  \\  
1/128 &9.50e-04	 &1.98	 &1.26e-01	 &1.00	 &9.41e-04	 &1.01  &1.07e-02	 &1.00 &3.56e-02	 &0.98  \\  
1/256 &2.39e-04	 &1.99	 &6.32e-02	 &1.00	 &4.68e-04	 &1.01  &5.37e-03	 &1.00 &1.79e-02	 &0.99  \\  
\hline
\end{tabular}
\end{table*}

\subsection{Experiment 2}
This experiment is the so-called Gresho--vortex problem that has been studied in \cite{FLNNS, HJL} and references therein for the isentropic flow. Initially, a vortex of radius $R_0=0.2$ is prescribed at location $(x_0,y_0) =(0.5,0.5)$ with the velocity field given by 
\begin{equation*}
r(x,y,0) = 1 , \quad 
\vU(x,y,0) =\left(\begin{array}{l}
u_R(R)*(y-0.5)/R\\
u_R(R)*(0.5-x)/R
\end{array} \right) ,
\end{equation*}
where $R=\sqrt{(x-0.5)^2+ (y-0.5)^2}$ and the radial velocity of the vortex $u_r$ is given by
\begin{equation*}
u_R(R)= \sqrt{\gamma}
\left\{ 
\begin{array}{ll}
2R/R_0 & \text{ if } 0 \leq R < R_0/2, \\
2(1-R/R_0) & \text{ if } R_0/2 \leq R < R_0, \\
0 & \text{ if } R \geq R_0. \\
\end{array}
\right.
\end{equation*}
As there is no analytical solution to this problem, we take the  solution  to  scheme \eqref{scheme} computed on the very fine mesh for $h=1/1024$ as the reference solution. 
We present the error of numerical solutions with respect to the reference solution in Table~\ref{tab_exp2}  for different values $\gamma$. 
Similarly as in Experiment 1, we see the second order convergence rate for the relative energy and the first order convergence rate for the density, velocity and the gradient of velocity.

\begin{table*}[h!]\centering
\caption{Experiment 2: error norms at $T=0.1$ for different $\gamma$}\label{tab_exp2}
\begin{tabular}{c|c|c|c|c|c|c|c|c|c|c}\hline 
  $h$   &  $e_{\mathfrak{E}} $ &  EOC & $ e_{\nabla \vu}$ &  EOC  &  $e_{\vr}$ &  EOC  &  $e_{\uu}$ &  EOC &  $e_{p }$ &  EOC  \\ \hline 
\multicolumn{11}{c}{ $\gamma=1.4$} \\ \hline 
1/32 &6.98e-04	 &--	 &6.09e-02	 &--	 &1.25e-05	 &--  &3.25e-03	 &-- &1.91e-04	 &--  \\  
1/64 &2.05e-04	 &1.77	 &3.39e-02	 &0.84	 &4.00e-06	 &1.65  &1.86e-03	 &0.80 &6.04e-05	 &1.66  \\  
1/128 &8.35e-05	 &1.29	 &2.08e-02	 &0.71	 &1.62e-06	 &1.31  &1.12e-03	 &0.73 &2.45e-05	 &1.30  \\  
1/256 &1.76e-05	 &2.25	 &9.37e-03	 &1.15	 &6.38e-07	 &1.34  &5.11e-04	 &1.14 &9.62e-06	 &1.35  \\ 
\hline 
\multicolumn{11}{c}{ $\gamma=1.67$} \\ \hline 
1/32 &8.38e-04	 &--	 &6.67e-02	 &--	 &1.37e-05	 &--  &3.56e-03	 &-- &2.50e-04	 &-- \\  
1/64 &2.47e-04	 &1.76	 &3.65e-02	 &0.87	 &4.41e-06	 &1.64  &1.97e-03	 &0.85 &7.93e-05	 &1.66  \\  
1/128 &7.29e-05	 &1.76	 &1.94e-02	 &0.91	 &1.76e-06	 &1.33  &1.06e-03	 &0.90 &3.17e-05	 &1.32  \\  
1/256 &1.80e-05	 &2.02	 &9.47e-03	 &1.04	 &7.04e-07	 &1.32  &5.16e-04	 &1.03 &1.26e-05	 &1.33  \\
\hline 
\multicolumn{11}{c}{ $\gamma=2$} \\ \hline 
1/32 &7.43e-04	 &--	 &6.34e-02	 &--	 &1.50e-05	 &--  &3.46e-03	 &-- &3.29e-04	 &--  \\  
1/64 &3.05e-04	 &1.29	 &3.99e-02	 &0.67	 &4.89e-06	 &1.61  &2.13e-03	 &0.70 &1.06e-04	 &1.64  \\  
1/128 &7.62e-05	 &2.00	 &1.99e-02	 &1.01	 &1.94e-06	 &1.33  &1.08e-03	 &0.99 &4.20e-05	 &1.33  \\  
1/256 &1.91e-05	 &1.99	 &9.73e-03	 &1.03	 &7.81e-07	 &1.31  &5.30e-04	 &1.02 &1.68e-05	 &1.32  \\ 
\hline 
\end{tabular}
\end{table*}

\begin{remark}
On one hand we have proven the theoretical convergence rate of $1/2$ for the relative energy functional (see Theorem~\ref{thm_error_estimates}). 
On the other hand we have observed the first order convergence rate for the density and velocity (see Table~\ref{tab_exp1}~and~\ref{tab_exp2}), 
and the second order convergence rate for the relative energy functional, as it is a function of the density and velocity squared.
Even if the theoretical convergence rate is not optimal, we would like to emphasize that it is  \emph{unconditional}, meaning there is no assumption on the regularity nor boundedness of the numerical solution. Moreover, as far as we know, it is the best theoretical convergence rate proven in the literature. 
Nevertheless,  assuming the numerical solution is bounded would allow the convergence rate to reach the value $2$ in the case of relative energy ($1$ for the density and velocity). 
\end{remark}

\section*{Conclusion}\label{9}

We have studied a finite difference scheme on the staggered grid for the multi-dimensional compressible isentropic Navier--Stokes equations in a periodic domain  originally proposed in \cite{HS_MAC}. 
The solutions of the scheme were shown to exist while preserving the positivity of the discrete density. Employing the stability and consistency estimates we have shown in Theorem~\ref{thm_convergence} that the numerical solutions of the scheme \eqref{scheme}  \emph{unconditionally} converge to a strong solution  of the limit system \eqref{strong_ns}  on its lifespan. 
 Further, we have  derived  \emph{uniform}  convergence rate for the error between the finite difference approximation and the corresponding strong solution in terms of the relative energy functional. Finally, we have  presented two numerical experiments to support our theoretical results.   
To the best of our knowledge, this is the first rigorous result concerning convergence analysis of a  finite difference scheme for the compressible  Navier--Stokes equations in the multi-dimensional setting.

{}

\bibliographystyle{plain}

\end{document}